 \documentclass[12pt]{article}

\usepackage[letterpaper, hmargin=1in, top=1in, bottom=1.2in, footskip=0.7in]{geometry}

\usepackage{rotating,pdflscape,xcolor,amssymb,subfigure,psfrag,amsmath,eufrak,bbm,epsfig,amsthm,mathtools,fancyhdr,graphicx}
\definecolor{vdarkred}{rgb}{0.6,0,0.2}
\definecolor{vdarkblue}{rgb}{0,0.2,0.6}
\usepackage[pdftex, colorlinks, linkcolor=vdarkblue,citecolor=vdarkred]{hyperref}
\usepackage{a4wide,amscd}
\usepackage{tikz} 
 \usepackage[usenames,dvipsnames]{pstricks}
 \usepackage{pst-grad} 
 \usepackage{pst-plot} 

\usepackage[small]{caption}

\newcommand{\lam}{\lambda}

\newcommand{\al}{\alpha}

\newcommand{\bet}{\beta}

\newcommand{\bgt}{\begin{itemize}}
\newcommand{\ent}{\end{itemize}}

\newcommand{\ds}{\displaystyle}

\newcommand{\brem}{\begin{rmk}}
\newcommand{\erem}{\end{rmk}}
\newcommand{\blem}{\begin{lem}}
\newcommand{\elem}{\end{lem}}
\newcommand{\bcor}{\begin{cor}}
\newcommand{\ecor}{\end{cor}}
\newcommand{\bTh}{\begin{Th}}
\newcommand{\eTh}{\end{Th}}
\newcommand{\bpropo}{\begin{propo}}
\newcommand{\epropo}{\end{propo}}

\newcommand{\op}{\operatorname}

\newcommand{\E}{\op{\mathbb{E}}}

\newcommand{\R}{\mathbb{R}}

\newcommand{\p}{\mathbb{P}}

\newcommand{\ff}{\frac{1}}

\newcommand{\eps}{\varepsilon}

\newcommand{\bbm}{\begin{bmatrix}}
\newcommand{\ebm}{\end{bmatrix}}
\newcommand{\bes}{\begin{equation*}}
\newcommand{\ees}{\end{equation*}}
\newcommand{\be}{\begin{equation}}
\newcommand{\ee}{\end{equation}}
\newcommand{\beqy}{\begin{eqnarray}}
\newcommand{\eeqy}{\end{eqnarray}}
\newcommand{\beq}{\begin{eqnarray*}}
\newcommand{\eeq}{\end{eqnarray*}}

\newcommand{\bpm}{\begin{pmatrix}}
\newcommand{\epm}{\end{pmatrix}}

\newtheorem{Th}{Theorem}[section]

\newtheorem{propo}[Th]{Proposition}
\newtheorem{Prop}[Th]{Proposition}

\newtheorem{lem}[Th]{Lemma}

\newtheorem{cor}[Th]{Corollary}

\theoremstyle{definition}
\newtheorem{rmk}[Th]{Remark}

\long\def\symbolfootnote[#1]#2{\begingroup
\def\thefootnote{\fnsymbol{footnote}}\footnote[#1]{#2}\endgroup} 
\providecommand{\keywords}[1]
{
	\small	
	\textbf{\textit{\footnotesize Keywords ---}} #1
}

\makeatletter
\def\@addpunct#1{%
	\relax\ifhmode
	\ifnum\spacefactor>\@m \else#1\fi
	\fi}
\newcommand{\keywordsname}{Key words}
\def\@setkeywords{%
	{\itshape \keywordsname.}\enspace \@keywords\@addpunct.}
\def\keywords#1{\def\@keywords{#1}}
\let\@keywords=\@empty
\g@addto@macro{\maketitle}{\begingroup%
	\let\@makefnmark\relax  \let\@thefnmark\relax%
	\ifx\@keywords\@mpty\else\@footnotetext{\@setkeywords}\fi%
	\endgroup}
\makeatletter
\def\blfootnote{\gdef\@thefnmark{}\@footnotetext}
\usepackage[affil-it]{authblk}
\date{}
\author{Cambyse Pakzad} 
\title{Poisson statistics at the edge of Gaussian $\beta$-ensemble at high temperature}

\newcommand{\Addresses}{{
		\bigskip
		\footnotesize
		\textsc{MAP 5, UMR CNRS 8145 - Universit\'e Paris Descartes, France}\par\nopagebreak
		\textit{E-mail address}: \texttt{cambyse.pakzad@gmail.com}
	}}
\begin{document}
	\maketitle
\blfootnote{\textup{2010} \textit{Mathematics Subject Classification}.
	60B20 - 60F05 - 60G70}
\begin{abstract}We study the asymptotic edge statistics of the Gaussian $\beta$-ensemble, a collection of $n$ particles, as the inverse temperature $\beta$ tends to zero as $n$ tends to infinity. In a certain decay regime of $\ds{\beta}$, the associated extreme point process is proved to converge in distribution to a Poisson point process as $n\to +\infty$. We also extend a well known result on Poisson limit for Gaussian extremes by showing the existence of an edge regime that we did not find in the literature.
\end{abstract}
	\section{Introduction}
	The study of spectral statistics in Random Matrix Theory has gathered a consequent volume of the research attention during the last decades. For several reasons, theses statistics are considered in the asymptotic regime: as the size of the matrix (and hence the number of eigenvalues) goes to infinity. One can inquiry about the behaviour of the whole spectrum (such as linear statistics), this is called \textit{global} statistics (or regime). The main object to study in this context is the empirical spectral measure and the goal  is to obtain a limiting distribution and give fluctuations around this limit. On the other hand, one can seek for more subtle, precise informations, like the spacing between two consecutive eigenvalues, or the nature of the largest eigenvalues; more generally, the joint distribution of eigenvalues in an interval of length $o(1)$. Such statistics are called \textit{local}. In this particular regime, we differentiate between the \textit{bulk} and the \textit{edge} statistics. The bulk regime focuses on intervals inside the support of the limiting spectral measure while the edge regime concerns about the boundary. In this article, we are mainly interested in the asymptotic local edge regime, which corresponds to the largest eigenvalues.
 
	Among random matrix models, two matrix ensembles are distinguished: Wigner matrices and invariant ensembles. The first one indicates matrices with independent components while the second gathers matrices whose law is invariant by symmetry group action. Their intersection is known as the GOE, GUE and GSE. Their origin trace back to the pioneer Wigner. He wanted to model complex highly correlated systems with (or lacking) different kind of symmetries (see \cite{Mehta,Forrester}) and considered Hamiltonians as large random matrices. The name stems from the invariance under certain group actions. The joint density of the eigenvalues can be derived (see for example Theorem $4.5.35$ on page $303$ in \cite{cupbook}, or in \cite{MatrixModelBetaEnsemble}) and is proportional to: $$P(\mathrm{d}\lam_1,\ldots,\mathrm{d}\lam_n) \propto \exp\left( -\frac{1}{2} \sum_{i=1}^{n}\lam^2_i\right)  \left| \Delta_n(\lam)\right| ^\beta \prod_{i=1}^{n}\mathrm{d}\lam_i .$$ The Vandermonde determinant is noted $\ds{\left| \Delta_n(\lam)\right| ^\beta:= \prod_{i<j}^{n} \left| \lam_j-\lam_i\right|^\beta}$, and $\beta\in \{1,2,4\}$.  
	 Let us mention that when $\beta=2$, the correlation functions, which will be our prime tool, describe a determinantal process (\textit{Gaudin-Mehta formula}, see for example Theorem $3.1.1$ on page $91$ in \cite{cupbook}). The idea that $\beta$ taking different values gives rise to different models is known as the \textit{Dyson's Threefold-Way} \cite{Dyson}.
 
	We can extend the model in two directions, allowing other values of $\beta$ and other potentials, by writing for $\beta>0$: \begin{align*}
P_{n,\beta,V}(\mathrm{d}\lam_1,\ldots,\mathrm{d}\lam_n):= \ff{Z_{n,\beta,V}} \exp\left( -{\beta} \sum_{i=1}^{n}V\left( \lam_i\right) \right)  \left| \Delta_n(\lam)\right| ^\beta \prod_{i=1}^{n}\mathrm{d}\lam_i .
	\end{align*} We refer this as the \textit{general $\beta$-ensemble}. If the potential is quadratic $\ds{V(x)=\frac{x^2}{4}}$, it reduces to the \textit{Gaussian $\beta$-ensemble} which is the object we study in this paper.
	
	In this context,  Dumitriu and  Edelmann \cite{MatrixModelBetaEnsemble} made a major breakthrough by constructing a matrix model for such $\beta$-ensemble with any $\beta>0$, hence extending the Dyson's Threefold-Way $\beta\in \{1,2,4\}$. It states that the Gaussian $\beta$-ensemble (viewed as a density probability function) is exactly the joint law of the spectrum of a certain simple matrix. The latter is obtained from successive Househ\"{o}lder transformations and has a symmetric tridiagonal
	form. 
	This representation of the Gaussian $\beta$-ensemble by a matrix model \cite{MatrixModelBetaEnsemble} led the way for many progresses \cite{Sutton1,VV,RRV,virag} on the understanding of the asymptotic local eigenvalue statistics for general $\beta>0$. In particular, the authors of \cite{Sutton1}, leaning on the tridiagonal structure of the Gaussian $\beta$-ensemble matrix model, gave multiple indications on how renormalized random matrices can be viewed as finite difference approximations to stochastic differential operators. Notably, the renormalization focuses on the top part of the matrix where the chi's random variables are large. This conjecture was investigated in \cite{RRV} where the properly renormalized largest eigenvalues are shown to converge jointly in distribution to the low-lying eigenvalues of a one-dimensional Schr\"{o}dinger operator, namely the stochastic Airy operator $\ds{\text{SAO}_\beta:=-\frac{\mathrm{d}^2}{\mathrm{d}x^2}+x+\frac{2}{\sqrt{\beta}}b'_x}$, understood as a random Schwartz distribution. The eigenvalues of this random operator, as for them, can be interpreted by variational formulation or by the eigenvalue-eigenvector equation between Schwartz distributions. Their result writes as for $k\geq 1$ fixed, denoting $\lam^\beta_1\geq \lam^\beta_2\geq ... \geq \lam^\beta_k$ the $k$ largest eigenvalues of the Gaussian $\beta$-ensemble matrix and $\Lambda^\beta_0\leq \Lambda^\beta_1 \leq ... \leq \Lambda^\beta_{k-1}$ the $k$ smallest eigenvalues of the stochastic Airy operator:
	$$n^{\frac{2}{3}} \left(2-\frac{\lam_i^\beta}{\sqrt{n}} \right) _{1\leq i \leq k}\xrightarrow[n\to\infty]{\text{law}}  \left(\Lambda^\beta_i \right)_{0\leq i \leq k-1}  . $$
	Since the minimal eigenvalue $\Lambda_0$ of SAO$_\beta$ has distribution minus TW$_\beta$, this work thereby enlarges Tracy-Widom law to all $\beta>0$, that is: $$n^{\frac{2}{3}}\left(\frac{\lam_i^\beta}{\sqrt{n}}-2 \right)\xrightarrow[n\to\infty]{\text{law}} \text{TW}_\beta.  $$
	The Tracy-Widom law (with parameter $\beta$) is qualified as \textit{universal}, in the sense that such local statistics hold for various matrix models (but also for objects outside of the random matrix field) and arises from highly correlated systems (such as modeled by some random matrices).
 
	For finite dimension $n$, one can choose $\beta = 0$ in the joint law $P$ of the Gaussian $\beta$-ensemble, which displays a lack of repulsion force as the Vandermonde factor vanishes, hence the correlation decreases, which means that randomness increases. In a Gibbs interpretation (which besides makes us refer to $Z_{n,\beta,V}$ and its counterparts as \textit{partition functions}), it comes down to consider an infinite temperature in such log-gas (terminology due to Dyson \cite{Dyson}). Readily, the joint density for $\beta=0$ is the density of $n$ i.i.d. Gaussian random variables whose maximum is known \cite{Resnick} to converge weakly, as $n\to +\infty$, when properly renormalized, to the Gumbel distribution, one of the three \textit{universal} distributions classes of the classical Extreme Value Theory. One deduces (see \cite[Th 7.1]{Coles}) Poisson limit for the Gaussian (ie: when $\beta=0$) extreme point process as the number $n$ of particles grows to infinity. This description in terms of Poisson point processes carries many informations and implies the limiting Gumbel distribution for the maximum particle.
 
	As the Gumbel law governs the typical fluctuations of the maximum of independent Gaussian variables, which corresponds to the case $\beta=0$, and the Tracy-Widom law stems from complicated (highly dependent) systems such as the largest particles in the case $\beta>0$ fixed and $n\to +\infty$, it is thus natural to ask for an interpolation between these two phases. The authors of \cite{TWHighTemp} answer this question by proving that the properly renormalized Tracy-Widom$_\beta$ converges in distribution to the Gumbel law as $\beta \to 0$. They use the characterization of the distribution of the bottom eigenvalues of the stochastic Airy operator in terms of the explosion times process of its associated Riccati diffusion (see \cite{RRV}). Regarding to our motivation, they could unfortunately not prove Poissonian statistics for the minimal eigenvalues $\ds{\left( \Lambda_i^\beta\right) }$, distributed according to the Tracy-Widom$_\beta$ law, in the limit $\beta \to 0$. This procedure would exactly reverse the order of the limits $\beta\to 0,n\to +\infty$ considered previously. Nonetheless, the authors investigated the weak convergence of the top eigenvalues in the double limit $\beta:=\beta_n \xrightarrow[n\to\infty]{}0$ by heuristic and numeric arguments. They alluded to the idea that one can achieve Poissonian statistics for $\beta$-ensemble using the same techniques as \cite{RRV,Sutton1}, at high temperature within the regime $n\beta \xrightarrow[n\to\infty]{} +\infty$. Concerning the bulk statistics, such work has been accomplished in the regime $\ds{\beta\sim n^{-1}}$, that is Poisson convergence of the point process $\ds{\sum_{i=1}^{n}\delta_{n(\lam_i-E)}}$ with $E \in (-2,2)$ an energy level in the \textit{Wigner sea} (see \cite{Duy1,Duy2}). This was achieved in \cite{PoissonTempLow} by means of correlation functions, which is also our method.
	 
	The goal of this paper is to understand the behaviour of the largest particles of the Gaussian $\beta$-ensemble as the inverse temperature $\beta_n$ converges to $0$ as $n$ goes to infinity. To this purpose, we study the limiting process of the extremes of the Gaussian $\beta$-ensemble. Among all possible decay rates for $\beta$, we restrict ourselves to the regime $\ds{n\beta \xrightarrow[n\to\infty]{}0}$. More precisely, our main result gives the convergence as $n\to +\infty$ of the extreme process toward a Poisson point process on $\R$. Two regimes for the extremes appear according to the asymptotic behaviour of a certain auxiliary scaling sequence $(\delta_n)$. In the situation where the latter converges, the scaling focuses on the very largest particles and the limiting process is inhomogeneous. Otherwise when $\ds{\delta_n\gg 1 }$, it comes down to consider the top particles which are slightly more inside the bulk. Also, it gives rise to a homogeneous limiting process.
	 Roughly speaking, the rescaled extreme eigenvalues approximate a Poisson point process which means that adjacent top particles are statistically independent. Our work also applies when $\beta$ is set to $0$ and \textit{de facto} includes asymptotics ($n\to +\infty$) of extremes of Gaussian variables ($\beta=0$). While the outcomes are identical for both $\beta$ cases, we want to stress out that the models are intrinsically distinct. We investigate this question in the subsequent Remark \ref{contigurmk}. Doing such simultaneous double scaling limit, we fulfill the corresponding task addressed by Allez and   Dumaz in \cite{TWHighTemp} within another regime mentioned in their work and by other means, namely, the correlation functions method used by Benaych-Georges and P\'ech\'e in \cite{PoissonTempLow}.
	
	 For $u=u_n$ and $v=v_n$ two sequences, we adopt the notation $\ds{u\ll v\iff \frac{u}{v} \xrightarrow[n\to\infty]{}0}$ and state our main result: 
	 \begin{Th}\label{main} Let $\beta=\beta_n$ be such that $\ds{0\leq \beta  \ll \ff{n\log( n)}}$. Let $(\lam_1,...,\lam_n)$ a family of random variables with joint law $P_{n,\beta}$:
		$$P_{n,\beta}(\mathrm{d}\lam_1,\ldots,\mathrm{d}\lam_n):= \ff{Z_{n,\beta}} \exp\left( -\frac{1}{2} \sum_{i=1}^{n}\lam^2_i\right)  \left| \Delta_n(\lam)\right| ^\beta \prod_{i=1}^{n}\mathrm{d}\lam_i,$$ with normalization constant $Z_{n,\beta}$ and Vandermonde determinant $\ds{\left| \Delta_n(\lam)\right| ^\beta:= \prod_{i<j}^{n} \left| \lam_j-\lam_i\right|^\beta}$.\ Given a positive sequence $(\delta_n)$, consider the extreme point process $$\xi_n :=\sum_{i=1}^n \delta_{a_n(\lam_i-b_n)},  $$ with modified Gaussian centering and scaling sequences:
	$$b_n:=\sqrt{2\log (n)}-\ff{2}\frac{\log\log (n) +2\log(\delta_n)+ \log (4\pi)}{\sqrt{2\log (n)}},\qquad a_n:= \delta_n \sqrt{2\log (n)} .$$
		\begin{itemize}
			\item Assume the perturbation $(\delta_n)$ to be such that $\ds{\delta_n \xrightarrow[n\to\infty]{}\delta>0}$. Then the random point process $\ds{\left( \xi_n \right) }$ converges in distribution to an inhomogeneous Poisson point process on $\R$ with intensity $\ds{e^{-\frac{x}{\delta}} \mathrm{d}x}$. 
			\item Assume $\ds{\delta_n \gg 1}$ such that $\ds{{\log(\delta_n)}{}\ll\sqrt{\log(n)}}$. Then the process $ \ds{\left( \xi_n \right) }$ converges in distribution to a homogeneous Poisson point process on $\R$ with intensity $1 $.
			\item When $\beta =0$, the condition on $(\delta_n)$ can be weakened to $\ds{{\log(\delta_n)}\ll \log(n)}$ in the previous statement.\\ 
		\end{itemize}
	\end{Th} 
	Let us first discuss the assumptions and conclusions of the theorem. 
We prove convergence of extreme point processes $$\mathcal{P}_n:=\sum_{i=1}^{n}\delta_{a_n(\lam_i-b_n)}$$ toward a Poisson point process on $\R$ with intensity $d\mu$ as $n\to +\infty$ for suitably chosen scaling sequences $(a_n),(b_n)$ and intensity $\mu$. This convergence occurs regardless to $\beta>0$ or $\beta=0$ although this gives rise to two different models. The scaling sequences are exactly the same in both cases and are derived from the classical Gaussian scaling (see \cite{Resnick}), except that we increase the scale $a_n$ by a multiplicative term $\delta_n$ and lower down the center $b_n$ by an additive term involving $\delta_n$.  We then observe two regimes:  first, when $$\ds{\delta_n \xrightarrow[n\to\infty]{}\delta>0},$$ the limiting process is an inhomogeneous Poisson process with intensity $\ds{e^{-\frac{x}{\delta}}dx}$ (which is a classical result in the purely Gaussian setting, when $\bet=0$). When $\delta_n\gg 1$, even in the purely Gaussian setting ($\bet=0$), we obtain a result that we did not find in the literature \cite{Coles,Leadbetter,Resnick}:  in this case, even though the interval considered (centered at $b_n$ and with width of order $a_n$) goes to $+\infty$, the limiting process is a homogeneous Poisson process.  An illustration of these phenomena is given in Figure \ref{figure} below.

		\begin{rmk}\label{contigurmk}
As previously mentionned, the Poissonian description of the extreme process, along with the normalizing constants $(a_n),(b_n)$ which display no dependence on $\beta$, is valid for both cases $\beta_n=0$ and $\beta_n>0$. The question of how close   both models are  is then raised. Therefore, we need to measure the impact of the  decay rate of $\beta_n$ upon the model. In this direction, one can compare the normalization constants $Z_{n,\beta}$ between different $\beta_n$ regimes. This idea emerges from equilibrium statistical mechanics where the $Z_{n,\beta}$ is seen as the \textit{partition function} in the Gibbs interpretation and is an important quantity characterizing the system. The computations show a transition at $\ds{\beta \sim n^{-2}}$. As soon as $\ds{\beta \gg n^{-2}}$, the repulsion is significant while for $\ds{\beta \ll n^{-2}}$, the partition function has same order as the normalizing constant of independent gaussians (case $\ds{\beta=0}$). We state this result in the forthcoming Lemma \ref{contigu} whose proof is postponed to Section \ref{section}. It indicates that our main theorem gains value when compelling $$n^{-2}\ll \beta\ll \left( n\log(n)\right)^{-1} ,$$ which corresponds to the regime where both models $\ds{\beta=0}$ and $\ds{n^{-2} \ll \beta\ll n^{-1}}$ are truly distinct. The critical role of $n^{2}$ in this description is consistent with the fact that one can write $$\ds{\log \left| \Delta_n(\lam)\right| ^\beta =\beta \sum_{i<j}^n \log \left| \lam_j-\lam_i\right|   }$$ with the sum having $n^2\left( 1+o(1)\right)$ terms. 
		\end{rmk}
		 \begin{figure}[!h]
		 	\centering	\includegraphics[width=15cm,height=5cm]{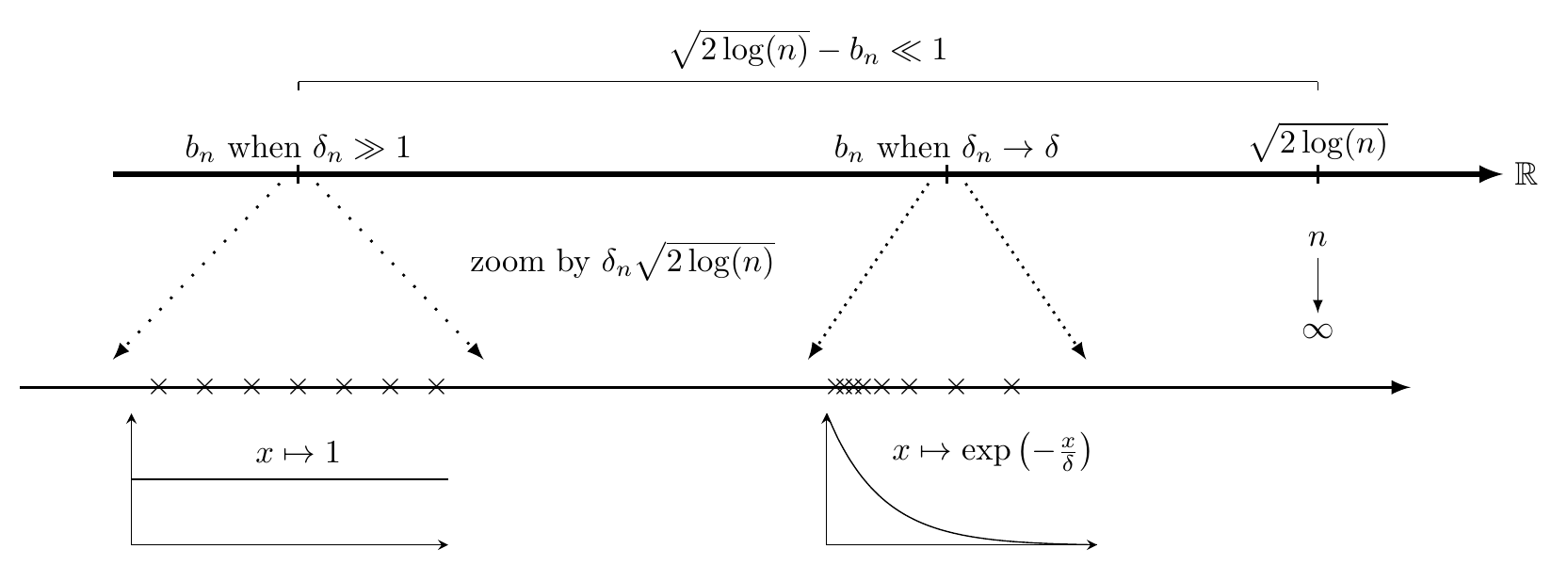}	
		 	\caption{\label{etiquette} The centering at $b_n$ for both cases  $\delta_n\to\delta >0$ and $\delta_n\to\infty$ are represented on the main line.   We zoom in around each $b_n$ by a factor $\ds{\delta_n \sqrt{2\log(n)}}$ and let $n$ go to $\infty$. For $\delta_n\gg 1$, the limiting object is a Poisson point process with intensity $1$. For $\ds{\delta_n\xrightarrow[n\to\infty]{}\delta}$, it leads to a Poisson point process with intensity $\ds{e^{-\frac{x}{\delta}}  }$.}
		 	\label{figure}
		 \end{figure} 
		\begin{lem} \label{contigu}
				Let $\beta \geq 0$ and $\beta'>0$.\begin{itemize}
					\item Assume $\ds{0\leq \beta\ll \beta'\ll \ff{n^2}}$, then $\ds{\frac{Z_{n,\beta'} }{Z_{n,\beta} }\xrightarrow[n\to \infty]{}1}$.
					\item Assume $\ds{0\leq \beta\ll \beta' \ll \ff{n} }$ and $\ds{\beta'\gg \ff{n^2}}$, then $\ds{Z_{n,\beta'}\ll Z_{n,\beta}}$.
				\end{itemize}
			\end{lem}
			The convergence toward a Poisson process for the extreme process is a much stronger information than the limiting distribution of the maximum. Indeed, one can deduce  the limiting distribution for the $k^{\text{th}}$ largest eigenvalue for fixed $ 1\leq k <+\infty$. 
			\begin{cor}\label{corGumbel}
Let $\beta=\beta_n$ be such that $\ds{0\leq \beta  \ll \ff{n\log( n)}}$. Let $(\lam_1,...,\lam_n)$ with joint law $P_{n,\beta}$. Let $(a_n),(b_n)$ from Theorem \ref{main} for $\delta=1$. Then, 	\begin{align*}
	P_{n,\beta}\left(a_n\left(\lam_{\max} - b_n \right) \leq x \right)&\xrightarrow[n\to\infty]{} \exp\left( -  \exp\left( -x\right) \right) .
	\end{align*}
			\end{cor}
	\begin{rmk}
			This result shows that we recover the Gumbel law as limiting distribution of the largest particle from the Poisson limit, so that we retrieve the result of \cite{TWHighTemp} corresponding to our setup. Besides, in view of \eqref{flo} in the next section, we know that the largest eigenvalue is unbounded when $n$ goes to infinity since the Gaussian distribution has unbounded support. In addition to this observation, our main result provides the explicit order and Gumbel fluctuations for the maximum eigenvalue.
		\end{rmk}
	The paper is organized as follows: first, we introduce and comment our model. To derive Poisson statistics, our method is the study of the correlation functions associated to the extreme point process. We refer to this as our \textit{main tool} and explain how it is exploited. Since the computations involve various estimates and quantities, we exhibit them as independent claims outside the main proof. The next section is devoted to the precise proof of our result. We give a tractable expression of the correlation functions. Then, we prove the conditions needed to provide inhomogeneous Poisson limit.  Our work transposes to the homogeneous limit with ease so we merge both cases in our statements. Finally, we give a peculiar proof of the statement when $\beta=0$. This is done by other means and displays a wider asymptotic regime for the perturbation $(\delta_n)$, so we present it as an independent result. 
		\begin{rmk}
			We consider two cases: $\delta_n = O(1)$ and $\delta_n \gg 1$. For the second case, the assumption required is $(\delta_n)$ such that $\ds{{\log(\delta_n)}\ll \sqrt{\log(n)}}$. It means that $\ds{\delta_n = e^{\eps_n \sqrt{2 \log(n)}}}$ with $\ds{\ff{\sqrt{\log(n)}}\ll \eps_n\ll 1}$. Note that the perturbation by $\delta_n$ corresponds to an increase of the zoom around the Gaussian center from first case minus a negligible factor.
			 Nonetheless, most of our results remain valid under both regimes and with a weaker growth restriction. For this reason, in this text, the reader will encounter a less restrictive hypothesis on $(\delta_n)$, namely $\ds{\log(\delta_n)\ll \log(n)}$. It ensures that $b_n$ is equivalent to $\ds{\sqrt{2\log(n)}}$ for any such $(\delta_n)$ as $n$ goes to infinity.
		\end{rmk}

	\begin{rmk}\label{extrafactor}
One may inquire about the extra factor $\ds{\log(n)}$ in our growth condition $\ds{n\beta\log(n)\ll 1}$ in comparaison with the original regime $\ds{n\beta\ll 1}$. Indeed, we also expect the result to hold when $\ds{\ff{n\log(n)}\ll \beta\ll \ff{n}}$. The reasons will become apparent along the paper. We will specially mention each time such restriction occurs. We also add that it seems rather difficult to overstep this technical limitation with our method.
	\end{rmk}

	\textbf{Acknowledgements:} I would like to express my gratitude to my supervisor Florent Benaych-Georges for his guidance throughout this work, his careful reading of the paper and the numerous advices he brought to me.
	\section{General model of the Gaussian $\beta$-ensemble for $\ds{\beta\ll 1}$ and $\al>0$}\label{sectionG}
	\subsection{Background and preliminaries}
		For any $\al >0$, $\beta\geq 0$, and $n\geq 1$, we define: \begin{align}
\label{definition_partition_function}Z_{n,\al,\beta}:= \int_{\R^n} \exp\left( -\frac{\al}{2} \sum_{i=1}^{n}\lam^2_i\right)  \left| \Delta_n(\lam)\right| ^\beta \prod_{i=1}^{n}\mathrm{d}\lam_i
		\end{align} with the Vandermonde determinant factor: $$  \left| \Delta_n(\lam)\right| ^\beta:=\prod_{i<j}^n \left|\lam_i-\lam_j \right|^\beta,$$ 
	and consider an exchangeable family $(\lam_1,...,\lam_n)$ of random variables with joint law  
	\begin{align}\label{P_nab}
P_{n,\al,\beta}(d\lam_1,...,d\lam_n)&:= \ff{Z_{n,\al,\beta}} \exp\left( -\frac{\al}{2} \sum_{i=1}^{n}\lam^2_i\right)  \left| \Delta_n(\lam)\right| ^\beta \prod_{i=1}^{n}\mathrm{d}\lam_i .
	\end{align} 
	When $\al=1$, we adopt the following notation: 
	$$Z_{n,\beta}:= \int_{\R^n} \exp\left( -\frac{1}{2} \sum_{i=1}^{n}\lam^2_i\right)  \left| \Delta_n(\lam)\right| ^\beta \prod_{i=1}^{n}\mathrm{d}\lam_i $$ 
	$$P_{n,\beta}(d\lam_1,...,d\lam_n):= \ff{Z_{n,\beta}} \exp\left( -\frac{1}{2} \sum_{i=1}^{n}\lam^2_i\right)  \left| \Delta_n(\lam)\right| ^\beta \prod_{i=1}^{n}\mathrm{d}\lam_i .$$ 
	In the sequel, the parameter $\al$ is always assumed to be $1$ except in some specific cases which will be mentionned. The reason of this choice shall be clear after incoming explanations.
	\begin{rmk}
For $\beta =0$, we retrieve the density of $n$ i.i.d. Gaussian random variables, which form a system of uncorrelated particles. The partition function in this case is just $\ds{Z_{n,\beta=0}=\left( 2\pi\right)^\frac{n}{2} }$. Allowing $\beta>0$, the Vandermonde factor vanishes when $\lam_i=\lam_j$ and acts as a (long range) repulsion force between the particles, which thereby constitutes a correlated system. The smaller $\beta$ is, the weaker repulsion operates. 
	\end{rmk}
	From the crucial matrix model of Dumitriu and Edelman \cite{MatrixModelBetaEnsemble}, we endow the Gaussian $\beta$-ensemble with a matrix structure. Recall that $\ds{\chi(k)}$ is defined for any $k>0$ by its density $\ds{\frac{  x^{\frac{k}{2}-1}e^{-\frac{x^2}{2}} }{ 2^{\frac{k}{2}-1}\Gamma(\frac{k}{2}) }}$ on $(0,+\infty)$. We state the corresponding result for our setup:
	\begin{Th}[Dumitriu, Edelman, \cite{MatrixModelBetaEnsemble}]
		Let $H:= H_{n,\al,\beta}$ the tridiagonal symmetric random matrix defined as:
		$$ \ff{\sqrt{\al}} \begin{pmatrix}
		g_1 &\ff{\sqrt{2}}X_{n-1} & && \\
		\ff{\sqrt{2}}X_{n-1}&g_2 & \ff{\sqrt{2}}X_{n-2}&&\\
		& \ff{\sqrt{2}}X_{n-2}&g_3 & \ff{\sqrt{2}}X_{n-3} &  \\
		& & \ddots &\ddots & \ddots &    \\
		& &\qquad  \ddots & \qquad \ddots &\ff{\sqrt{2}}X_{1} &   \\
		& & &\ff{\sqrt{2}}X_1& g_n 
		\end{pmatrix} ,$$ with $\ds{(g_i)_{1\leq i \leq n}\sim \mathcal{N}(0,1)}$ i.i.d. sequence, $\ds{(X_i)_{1\leq i \leq n-1}}$ an independent sequence such that $\ds{X_i \sim \chi(i\beta)}$ and independent overall entries up to symmetry. \\ \\
		For any $\al>0$, $\beta\geq 0$, the joint law of the eigenvalues $(\lam_1,...,\lam_n)$ of $H$ is $P_{n,\al,\beta}$.
	\end{Th}
	It makes the connection between the particles of law $P_{n,\al,\beta}$ and the spectrum of $H$. By trace invariance, we can easily access further information: when $\ds{\al \sim 1+\frac{n\beta}{2}}$, the empirical spectral distribution $\ds{ L_n:= \ff{n}\sum_{i=1}^{n}\delta_{\{\lam_i\}} }$ of $H_{n,\al,\beta}$ has asymptotic first moment $0$ and second moment $1$. This convergence motivates the choice $\al\sim 1$ when $\ds{\beta \ll \ff{n}}$. 
 
	In \cite{PoissonTempLow}, with the choice $\ds{\al\sim 1+\frac{n\beta}{2}}$, the authors proved under the assumption of simultaneous limit $\ds{n\beta_n\xrightarrow[]{}2\gamma}$ as $n\to +\infty$, a continuous asymptotic interpolation for the empirical spectral measure between the Wigner semicircle law $(\gamma\to +\infty)$ and the Gaussian distribution $(\gamma=0)$. The latter case is of our interest and particularly to the setting $\ds{\beta \ll \ff{n}}$, they proved that: \begin{align}
\label{flo}\ff{n}\sum_{i=1}^{n}f(\lam_i) \xrightarrow[n\to\infty]{\p} \int_{\R} \ff{\sqrt{2\pi}}f(x) e^{-\frac{x^2}{2}} \mathrm{d}x,\qquad \qquad \forall f\in \mathcal{C}_b(\R).
	\end{align} This convergence also justifies the choice $\al=1$ in our model for $\ds{\beta \ll \ff{n}}$. 
	
	Such transition from Gaussian to Wigner distribution is furthermore investigated in \cite{allez2012invariant,Duy1}. The limiting empirical eigenvalue density in the double limit $\ds{n\beta_n\xrightarrow[]{}2\gamma\geq 0}$ is derived as a family of densities with parameter $\gamma\geq 0$. Each of the papers \cite{allez2012invariant,PoissonTempLow,Duy1} although provides different computations, hence giving rise to new identities which seem difficult to prove directly.
	\subsection{Correlation functions and Poisson convergence}
	The theorem we intend to prove will stem from the following result, which thereby makes it the cornerstone of our demonstration. It ensures that under pointwise convergence of the correlation functions and some uniform bound on it, the initial point process converges to a Poisson process. 
	\begin{Prop}[Benaych-Georges, P\'ech\'e, \cite{PoissonTempLow}]\label{tool}
		Let $X$ be a locally compact Polish space and $\mu$ a Radon measure on $X$. Let $(\lam_1,...,\lam_n)$ be an exchangeable random vector taking values in $X$ with density $\rho_n$ with respect to $\mu^{\otimes n}$. For $1\leq k \leq n$, we define the $k$-th correlation function on $X^k$: \begin{align}\label{def_corfct_tool}
R^n_k(x_1,...,x_k):=\frac{n!}{(n-k)!}\int_{(x_{k+1},...,x_n)\in X^{n-k}}\rho_n(x_1,...,x_n)\mathrm{d}\mu^{\otimes(n-k)}(x_{k+1},...,x_{n}).
		\end{align}
		Suppose there exists $\theta\geq 0$ independent of $n$ such that: 
\begin{itemize}
		\item For $1\leq k<+\infty$ fixed integer, on $X^k$, we have the pointwise convergence: 
		\begin{align*}
R^n_k(x_1,...,x_n)\xrightarrow[n\to\infty]{}\theta^k. 
		\end{align*}
		\item For each compact $K\subset X$, there exists $\theta_K>0$ such that for any integer $n\geq 1$ large enough, any integer $k\geq 1$, on $K^k$, we have: \begin{align*}
1_{\{k\leq n\}} R^n_k(x_1,...,x_k)&\leq \theta^k_K.
		\end{align*}
		Then, the point process $\ds{\mathcal{P}_n:=\sum_{i=1}^{n}\delta_{\lam_i}}$ converges in distribution to a Poisson point process with intensity $\theta \mathrm{d}\mu$ as $n\to +\infty$.
		
		$\ds{}$
\end{itemize}		
	\end{Prop}
\begin{rmk}
The proof can be found in \cite[Prop.5.6]{PoissonTempLow} where the scheme is successfuly applied to the bulk regime when $\ds{n\beta \xrightarrow[n\to\infty]{}{2\gamma }\geq 0}$. In this paper, we inspect the edge regime by using Proposition \ref{tool} for the rescaled $\beta$-ensemble $\ds{\left( a_n\left(\lam_i-b_n \right) \right)_{1\leq i \leq n} }$ in two ways as we chose the real measure $\mu$ to be $\ds{e^{-\frac{x}{\delta}}\mathrm{d}x}$ or just the Lebesgue measure $\lam$. The two induced densities $\rho_n$ differ only by a $\delta$-dependent term. By integral linearity, the same goes for the correlation functions. We derive the mandatory conditions in both cases, leading to two types of Poisson limit, but proofs are similar. 
	\end{rmk}
	\subsection{Partition functions}\label{section}
	In this section, we list some identities, bounds and asymptotics involving partition fonctions. They will be used from time to time in the sequel of the text.

 First, we give the main formula for the partition functions. From this, we will be able to compute several asymptotics of partition functions ratio.
 
	\begin{lem}\label{fondamental_partition_function_lem} For any $\al,\beta>0$ and $n\geq 1$, the following identity holds: \begin{align}
\label{fondamental_partition_function_eq1}Z_{n,\al,\beta} = (2\pi)^{\frac{n}{2}}(n!) \al^{-\beta\frac{n(n-1)}{4}-\frac{n}{2}} \prod_{i=0}^{n-1}\frac{\Gamma\left(\left( i+1\right)  \frac{\beta}{2}\right) }{\Gamma\left( \frac{\beta}{2}\right) } .
		\end{align}
				If $\beta\geq 0$, one has also: 
		\begin{align}\label{fondamental_partition_function_eq2}
		Z_{n,\al,\beta} &=(2\pi)^{\frac{n}{2}} \al^{-\beta\frac{n(n-1)}{4}-\frac{n}{2}} \prod_{i=1}^{n}\frac{\Gamma\left(1+  \frac{i\beta}{2}\right) }{\Gamma\left(1+ \frac{\beta}{2}\right) }.\end{align}
	\end{lem}
	\begin{proof}
		Let $\beta>0$. By the Selberg integral theorem in \cite{cupbook}, we have:  $$\int_{\R^n}   \left| \Delta_n(x)\right| ^\beta e^{-\ff{2}\sum_{i=1}^{n}x^2_i}\mathrm{d}x_1 \cdots \mathrm{d}x_n= (n!)(2\pi)^\frac{n}{2} \prod_{i=0}^{n-1} \frac{\Gamma\left(\left( i+1\right)  \frac{\beta}{2}\right) }{\Gamma\left( \frac{\beta}{2}\right) }  .$$
		By the change of variable $x_i = y_i\sqrt{\al}$, we get the fundamental identity on partition functions \begin{align*}
		Z_{n,\al,\beta} &= (2\pi)^{\frac{n}{2}}(n!) \al^{-\beta\frac{n(n-1)}{4}-\frac{n}{2}} \prod_{i=0}^{n-1}\frac{\Gamma\left(\left( i+1\right)  \frac{\beta}{2}\right) }{\Gamma\left( \frac{\beta}{2}\right) }
		\\&=(2\pi)^{\frac{n}{2}} \al^{-\beta\frac{n(n-1)}{4}-\frac{n}{2}} \prod_{i=1}^{n}\frac{\Gamma\left(1+  \frac{i\beta}{2}\right) }{\Gamma\left(1+ \frac{\beta}{2}\right) }.
		\end{align*} The case $\beta=0$ is easily treated.
	\end{proof}
	We are now ready to prove several results needed later.
	
	\begin{lem}\label{partition_function_0}
		Fix an integer $1\leq k<+\infty$ and a real number $0<\al<+\infty$. Let $\beta\geq 0$ such that $n\beta \ll 1$. Let any $(\delta_n)$ positive real sequence such that $\ds{ {\log(\delta_n)}{}\ll \log(n) }$. Then,
		\begin{align}\label{partition_function_01}
\frac{Z_{n-k,\al,\beta}}{Z_{n,\al,\beta}} &=\left( 1+o(1)\right)  \left( 2\pi \right)^{-\frac{k}{2} } \al^{\frac{k}{2}}  
		\end{align}
		\begin{align}\label{partition_function_02}
\frac{Z_{n-k,\al-\frac{k\beta}{4b^2_n},\beta}}{Z_{n-k,\al,\beta}}&=1+o(1).
		\end{align}
	\end{lem}
	\begin{proof}
		For $u\ll  1$, recall the equivalence of the Gamma function near the origin: $$\Gamma(u)=\ff{u}\left( 1+o(1)\right) \gg 1. $$ 
		Using equation (\ref{fondamental_partition_function_eq1}) of Lemma \ref{fondamental_partition_function_lem}, we compute the ratio (\ref{partition_function_01}) for $\beta>0$: \begin{align*}
		\frac{Z_{n-k,\al,\beta}}{Z_{n,\al,\beta}} &= \left( 2\pi \right)^{-\frac{k}{2} } \frac{(n-k)!}{n!} \al^{\frac{k}{2}+\frac{\beta}{4}\left(2nk -k(k+1) \right) } \prod_{i=n-k}^{n-1}\frac{ \Gamma\left( \frac{\beta}{2}\right)}{\Gamma\left(\left( i+1\right)  \frac{\beta}{2}\right)}
		\\&= \left( 1+o(1)\right)  \left( 2\pi \right)^{-\frac{k}{2} } \al^{\frac{k}{2}}.
		\end{align*}
		If $\beta =0$, the identity claimed is readily computed from equation (\ref{fondamental_partition_function_eq2}).
		
		Let us show the asymptotic (\ref{partition_function_02}).
		For $\al>0$, using (\ref{fondamental_partition_function_eq1}), we have:
		\begin{align*}
		Z_{n-k,\al-\frac{k\beta}{4b^2_n},\beta} &= (2\pi)^{\frac{n-k}{2}}(n-k!)\al^{-\beta\frac{(n-k)(n-k-1)}{4}-\frac{n-k}{2}} \left(1-\frac{k\beta}{4\al b^2_n}\right)^{-\beta\frac{(n-k)(n-k-1)}{4}-\frac{n-k}{2}} \prod_{i=0}^{n-k-1}\frac{\Gamma\left(\left( i+1\right)  \frac{\beta}{2}\right) }{\Gamma\left( \frac{\beta}{2}\right) }.
		\end{align*}
		Thus by a Taylor expansion of $\ds{x\mapsto \log(1-x)}$ around $0$:
		\begin{align*}
		\frac{Z_{n-k,\al-\frac{k\beta}{4b^2_n},\beta}}{Z_{n-k,\al,\beta}}&= \exp\left(  \left( -\beta\frac{(n-k)(n-k-1)}{4}-\frac{n-k}{2}\right)\left( -\frac{k\beta}{4\al b^2_n}+O\left( -\frac{k\beta}{4\al b^2_n}\right)^2  \right) \right) .
		\end{align*}
		The last term converges to $1$ under our hypothesis. The case $\beta=0$ is easily treated.
	\end{proof}
	\begin{lem}\label{partition_function_1}Assume $\ds{n\beta\ll 1}$. Let any $(\delta_n)$ positive real sequence such that $\ds{ {\log(\delta_n)}{}\ll \log(n) }$. Fix a positive real number $0<\al<+\infty$. There exists a sequence $(c_n)$ converging to $1$, such that for $n$ large enough,  \begin{align}\label{partition_function_1_eq1}
		\frac{Z_{n-1,\al b^2_n-\frac{\beta}{4},\beta}}{Z_{n,\al,\beta}} &\leq  c_n \sqrt{\frac{{\al}}{{2\pi}} } \left( b_n\right) ^{-\beta\frac{(n-1)(n-2)}{2}-n+1}.
		\end{align}
\end{lem}
\begin{proof} Note that our assumptions imply that the partition function $\ds{Z_{n-1,\al b^2_n-\frac{\beta}{4},\beta}}$ is well defined since $\ds{\al b^2_n-\frac{\beta}{4}>0}$. 
	From the identity (\ref{fondamental_partition_function_eq2}), we have: \begin{align*}
	Z_{n-1,\al b^2_n-\frac{\beta}{4},\beta}&=(2\pi)^{\frac{n-1}{2}} \left( \al b^2_n-\frac{\beta}{4} \right) ^{-\beta\frac{(n-1)(n-2)}{4}-\frac{n-1}{2}} \prod_{i=1}^{n-1}\frac{\Gamma\left(1+ \frac{i\beta }{2}\right) }{\Gamma\left( 1+\frac{\beta}{2}\right) }
	\end{align*}
	and we can compute the ratio:
	\begin{align*}
	\frac{Z_{n-1,\al b^2_n-\frac{\beta}{4},\beta}}{Z_{n,\al,\beta}}&=\left(  b_n\right) ^{-\beta\frac{(n-1)(n-2)}{2}-n+1} \al^{\frac{(n-1)\beta}{2}+\ff{2}}\left( 1-\frac{\beta}{4\al b^2_n} \right) ^{-\beta\frac{(n-1)(n-2)}{4}-\frac{n-1}{2}} \frac{\Gamma\left( 1+\frac{\beta}{2}\right) }{\sqrt{2\pi}\Gamma\left(1+\frac{n\beta}{2} \right) }.
	\end{align*}
	We apply the following inequality: \begin{align}\label{ineq4^x}
 \ff{1-x}\leq 4^x, \qquad x\in [0,\ff{2}],
	\end{align} with $\ds{x=\frac{\beta}{4\al b^2_n} \leq\ff{2} \iff \frac{\beta}{2}\leq \al b^2_n}$. This inequality is true when $n$ is large enough. Thus,
	\begin{align*}
	\frac{Z_{n-1,b^2_n-\frac{\beta}{4},\beta}}{Z_{n,\beta}}&\leq c_n \frac{\Gamma\left( 1+\frac{\beta}{2}\right) }{\sqrt{2\pi}\Gamma\left(1+\frac{n\beta}{2} \right) }  \left(  b_n\right) ^{-\beta\frac{(n-1)(n-2)}{2}-n+1} \sqrt{\al},
	\end{align*} where we have set:		$$ 		c_n :=\exp\left(\left(  {\frac{\beta^2 (n-1)(n-2)}{16\al b^2_n}+\frac{\beta(n-1)}{8\al b^2_n}}\right)  \log 4+\frac{(n-1)\beta}{2}\log \al \right).$$ It is clear that the latter sequence converges to $1$ from our hypothesis on $\beta$. 
	
	Besides, the Gamma function has local minimum at $\sim 0.8$ with value $\approx 1.44$, it follows that for $\beta\ll 1$, $$\frac{\Gamma\left( 1+\frac{\beta}{2}\right) }{\sqrt{2\pi}\Gamma\left(1+\frac{n\beta}{2} \right) } \leq \frac{\Gamma\left( 1+\frac{\beta}{2}\right) }{\sqrt{2\pi}}\leq \frac{\Gamma\left( 2\right) }{\sqrt{2\pi} }  = \ff{\sqrt{2\pi}}.$$
\end{proof}
The next result states uniform bounds over $k\leq n$ for ratios of partition functions in connection with second condition of Proposition \ref{tool}.
	\begin{lem}\label{partition_function_3} Let $\beta\geq 0$ and $(\delta_n)$ positive real sequence such that $\ds{ {\log(\delta_n)}\ll \log(n) }$. Assume $\ds{\beta \ll \ff{n}}$. Let $k$ an integer such that $1\leq k\leq n$. Then for $n$ large enough,
		\begin{align}\label{partition_function_31}
\frac{Z_{n-k,1-\frac{k\beta}{4b^2_n},\beta}}{Z_{n-k,\beta}}&\leq 4^k 
		\end{align}
		\begin{align}\label{partition_function_32}
\frac{Z_{n-k,\beta}}{Z_{n,\beta}} &\leq \left(\sqrt{ \frac{2}{{\pi}}}\right)^k.
		\end{align}
	\end{lem}
	\begin{proof} Since the case $\beta=0$ can be easily treated, we only consider $\beta>0$.
		From our hypothesis, $k\beta$ is less than $1$ when $n$ is large enough and:
		$$\frac{k\beta}{4b^2_n}\leq \ff{2}\iff  k\beta \leq 2b^2_n \qquad \text{which is true.}$$
		$$\frac{\beta}{4b^2_n}\left( \beta\frac{(n-k)(n-k-1)}{4}+\frac{n-k}{2}\right)\leq \frac{\left( n\beta\right)^2 }{16b^2_n}+\frac{n\beta}{8b^2_n}\leq 1. $$
		Using (\ref{fondamental_partition_function_eq2}) to compute the ratio and applying inequality (\ref{ineq4^x}), we have:
		\begin{align*}
		\frac{Z_{n-k,1-\frac{k\beta}{4b^2_n},\beta}}{Z_{n-k,\beta}}&= \left(1-\frac{k\beta}{4 b^2_n}\right)^{-\beta\frac{(n-k)(n-k-1)}{4}-\frac{n-k}{2}}
		\\&\leq \exp\left( {\frac{k\beta}{4 b^2_n}\left( \beta\frac{(n-k)(n-k-1)}{4}+\frac{n-k}{2}\right)\log\left(4 \right)  }\right)  
		\leq 4^k. 
		\end{align*}
		We prove the second statement. From the identity (\ref{fondamental_partition_function_eq2}) with $\al=1$, we get: \begin{align*}
		\frac{Z_{n-k,\beta}}{Z_{n,\beta}} &= \left( 2\pi \right)^{-\frac{k}{2} } \prod_{i=n-k+1}^n \frac{\Gamma(1+\frac{\beta}{2})}{\Gamma(1+\frac{i\beta}{2})} .
		\end{align*}
		The Gamma function has local minimum at $\approx 1.46$ with value $\approx 0.8$, it follows that for any $i\leq n$, since $\beta\ll 1$, $$\ff{2}\leq \Gamma\left(1+\frac{i\beta}{2} \right)  \leq \Gamma\left(1+\frac{\beta}{2}  \right)\leq 1.$$ Hence, $$\prod_{i=n-k+1}^n \frac{\Gamma(1+\frac{\beta}{2})}{\Gamma(1+\frac{i\beta}{2})}\leq 2^k .$$
	\end{proof}
		At last, we prove the previously stated lemma which compares the partition functions between different regime of $\beta$: 
			\begin{proof}[Proof of Lemma \ref{contigu}]
				Denoting $\gamma$ the Euler constant, recall that for $x\ll 1$: $$\log \Gamma \left( 1+x\right)  = -\gamma x +\frac{\pi^2}{12}x^2 +o(x^3). $$
				Remark that for any $k\geq 1$, one has: $$ n^k \beta^{k-1}\gg n^{k+1}\beta^k. $$
				We compute the ratios:
				\begin{align*}
				\frac{Z_{n,\beta}}{Z_{n,0}}&= \prod_{i=1}^{n}\frac{\Gamma\left( 1+\frac{i\beta}{2}\right) }{\Gamma\left( 1+\frac{\beta}{2}\right)}
				= \exp\left( -n\log \Gamma \left(1+\frac{\beta}{2} \right)+\sum_{i=1}^{n}\log \Gamma \left(1+\frac{i\beta}{2} \right)  \right) 
				\end{align*}
and,
\begin{align*}
\frac{Z_{n,\beta}}{Z_{n,\beta'}}&= \exp\left( -\sum_{i=1}^{n}\log \Gamma \left(1+\frac{i\beta'}{2} \right) +\sum_{i=1}^{n}\log \Gamma \left(1+\frac{i\beta}{2} \right) + n\log \frac{\Gamma\left( 1+\frac{\beta'}{2}\right)}{\Gamma\left( 1+\frac{\beta}{2}\right)}\right) .
\end{align*}
				Since the quantity $i\beta$ converges to $0$ uniformly in $i\leq n$, we deduce that:\begin{align*}
				\sum_{i=1}^{n}\log \Gamma \left(1+\frac{i\beta}{2} \right) &= -\frac{\gamma \beta}{2} \sum_{i=1}^{n}i +\frac{\pi^2}{48}\beta^2 \sum_{i=1}^{n}i^2 +nO\left( n^3\beta^3 \right) 
				\\&=-\frac{\gamma \beta}{8}(n^2 + n)+\frac{\pi^2 }{48}\beta^2\left(\frac{n^3}{3}+\frac{n^2}{2} +\frac{n}{6}\right) +O\left( n^4\beta^3 \right) 
				\\&= -\frac{\gamma }{8}n^2 \beta \left( 1+o(1)\right) .
				\end{align*}
				Besides, the log-Gamma expansion and hypothesis $n\beta \vee n\beta' \ll 1$ imply that: $$ n\log {\Gamma\left( 1+\frac{\beta'}{2}\right)}-n\log{\Gamma\left( 1+\frac{\beta}{2}\right)} \xrightarrow[n\to +\infty]{} 0. $$
			    We deduce that:  $$
				\frac{Z_{n,\beta}}{Z_{n,0}}= \exp\left(- \frac{\gamma }{8}n^2 \beta   \left( 1+o(1)\right) \right) 
			,\qquad 
				\frac{Z_{n,\beta}}{Z_{n,\beta'}}= \exp\left( \frac{\gamma }{8}n^2 \left( \beta'-\beta\right)  \left( 1+o(1)\right) \right) 
				.$$ The claims readily follow.
			\end{proof}
	\subsection{Estimates: bulk and largest eigenvalues}
	In the section, we establish some estimates on the eigenvalues $(\lam_i)_{1\leq i\leq n}$ of $H_{n,\al,\beta}$, which are $P_{n,\al,\beta}$-distributed. Since the particles are exchangeable, every estimate will concern $\lam_1$. \\ \\ We give exponential type bound on the probability of a scaled eigenvalue to be larger than any arbitrary value. Same-wise, an exponential estimate for the probability of $\lam_1$ to be as close as we want to any value is given.
 
	These estimates will be crucial for the analysis of the integral term $\tilde{R}^n_k$, which presents itself as the expectation of some functional of $(\lam_i)$. The link is made through to the identity $$\E\left|X \right|= \int_{0}^{+\infty}\p\left(\left|X\right|\geq t \right) \mathrm{d}t.$$ 
	We begin with a technical but fundamental lemma. 
	\begin{lem}\label{lem_technique}
		For any $a,b\in\R$ and $\beta>0$, one has: \begin{align}
\label{bound_technique}\left| a+b\right|^\beta \leq 2^\beta e^{\beta \frac{a^2 + b^2}{8}} .
		\end{align}
	\end{lem}
	\begin{proof}
		First recall two inequalities: $$\left| x\right| \leq 2e^{\frac{x^2}{16}}, \qquad (x+y)^2\leq 2x^2 + 2y^2. $$
		Applying the first inequality with $x=a+b$, then using the second one give:
		\begin{align*}
		\left| a+b\right|^\beta &\leq \left( 2e^{\frac{(a+b)^2}{16}} \right)^\beta
		\leq  \left( 2e^{\frac{a^2+b^2}{8}} \right)^\beta.
		\end{align*}
	\end{proof}
	This inequality is of interest because it roughly allows to gain quadratic sum bound $a^2+b^2$ from a quantity of type $\log \left| a+b\right|$. It provides an useful algebraic mean to upper-bound the integral term $\tilde{R}^n_k$ with a ratio of partition functions.
 
	Next, we show an estimate on the scaled top eigenvalue. This result is also established in \cite{PoissonTempLow} but in another form, more appropriate to the bulk regime. For the sake of completeness, we give its proof since our version is slightly different.
	\begin{lem}\label{top_estimate}
	Let $M>0$ such that $\al\vee n\beta \leq M$. There exists a constant $C_M>0$ such that for any $\beta,t>0$, $u\in \R$ and any $n\geq 1$ large enough,  \begin{align}\label{top_estimate_eq}
	P_{n,\al,\beta}\left(\left|u-\frac{\lam_1}{b_n} \right| \geq t \right)&\leq  C_M b_n^{\beta(n-1)-2} \frac{\exp\left( -\frac{\al}{2}\left( b^2_n-\frac{n\beta}{4\al}\right)^2 \left( t+\frac{b^2_n u}{b^2_n-\frac{n\beta}{4\al}} \right)^2\right) }{\al \left( t+\frac{b^2_n u}{b^2_n-\frac{n\beta}{4\al}} \right) }.
	\end{align} 
	\end{lem}
	\begin{proof}
		Let $u\in \R$. Let $(\lam_1,...,\lam_n)$ an exchangeable family of random variables distributed according to $P_{n,\al,\beta}$. By a change of variable in (\ref{P_nab}), the family $\ds{ \left(\frac{\lam_i}{b_n}-u \right)_{1\leq i \leq n}   }$ has law: 
		$$ \frac{b_n^{n+\beta\frac{n(n-1)}{2}}}{Z_{n,\al,\beta}} \left| \Delta_n(z)\right| ^\beta e^{-\frac{\al}{2} b^2_n \sum_{i=1}^n \left( z_i+u\right)^2  } \mathrm{d}z_1 \cdots \mathrm{d}z_n.$$
		Now for $t>0$, the quantity $\ds{	\Lambda_{n,t,u}:=P_{n,\al,\beta}\left(\left|u-\frac{\lam_1}{b_n} \right| \geq t \right)}$ equals to: \begin{align*}
		  \frac{b_n^{n+\beta\frac{n(n-1)}{2}}}{Z_{n,\al,\beta}}\int_{\left| z_1\right| \geq t }\int_{\R^{n-1}}\prod_{j=2}^{n} \left|z_1-z_j \right|^\beta e^{ -\frac{\al b^2_n}{2}   \left( z_1+u\right)^2  }    \left| \Delta_{n-1}(z_2,...,z_n)\right| ^\beta  e^{-\frac{\al}{2} b^2_n \sum_{i=2}^n \left( z_i+u\right)^2  } \mathrm{d}z_1 \cdots \mathrm{d}z_n .
		\end{align*}
		The product term in the first integral involves every variables. We split this overlapping term thanks to the fundamental inequality (\ref{bound_technique}) of Lemma \ref{lem_technique}. It leads to:
		\begin{align*}
		\Lambda_{n,t,u}&\leq  \frac{b_n^{n+\beta\frac{n(n-1)}{2}}2^{n\beta}}{Z_{n,\al,\beta}}\int_{\left| z_1\right| \geq t } \exp\left(\frac{n\beta}{8}z^2_1 -\al \frac{b^2_n}{2}(z_1+u)^2 \right)   \mathrm{d}z_1 \times 
		\\ &   \times \int_{\left( z_2,...,z_n\right) \in \R^{n-1}}\left| \Delta_{n-1}(z_2,...,z_n)\right| ^\beta  \exp\left(\frac{\beta}{8}\sum_{i=2}^{n}z_i^2 - \al \frac{b^2_n}{2}\sum_{i=2}^{n}(z_i+u)^2 \right) \mathrm{d}z_2...\mathrm{d}z_n .
		\end{align*}
		The first integral term will be linked to a Gaussian tail and the second to a partition function. For this, we need to complete the square.

		Using the two following algebraic identities:
		\begin{align*}
		\frac{\beta}{8}\sum_{i=2}^{n}z_i^2 - \al \frac{b^2_n}{2}\sum_{i=2}^{n}(z_i+u)^2&= -\frac{\al}{2}\left(b^2_n-\frac{\beta}{4\al} \right)\sum_{i=2}^{n}\left(z_i +\frac{b_n^2 u }{b^2_n -\frac{\beta}{4\al}}  \right)^2 \\&\qquad  +\al \frac{b^4_n u^2}{2\left( b^2_n-\frac{\beta}{4\al}\right) } (n-1)-\al\frac{b_n^2 u^2}{2}(n-1) 
		\end{align*}
		$$\frac{n\beta}{8}z^2_1 -\al \frac{b^2_n}{2}(z_1+u)^2 =-\frac{\al}{2}\left(b^2_n-\frac{n\beta}{4\al} \right) \left(z_1+\frac{b_n^2 u}{b_n^2-\frac{n\beta}{4\al}} \right)^2 +\al\frac{b_n^4u^2}{2\left(b_n^2-\frac{n\beta}{4\al} \right) }  -\al\frac{b_n^2 u^2}{2},$$
		we can write:
		\begin{align}\label{Lambda_GZ}
		\Lambda_{n,t,u}&\leq  \frac{b_n^{n+\beta\frac{n(n-1)}{2}}2^{n\beta}}{Z_{n,\al,\beta}}e^{\al\frac{b^4_nu^2}{2\left(b_n^2-\frac{n\beta}{4\al} \right) }  -\al\frac{b^2_n u^2}{2}} G(t)Z.
		\end{align}
		where $$G(t):=\int_{\left| z_1\right| \geq t } \exp\left(-\frac{\al}{2}\left(b^2_n-\frac{n\beta}{4\al} \right) \left(z_1+\frac{b_n^2 u}{b_n^2-\frac{n\beta}{4\al}} \right)^2 \right)   \mathrm{d}z_1 $$
		\begin{align*}
		Z &:=e^{\al\frac{b^4_n u^2 (n-1)}{2\left( b^2_n-\frac{\beta}{4\al}\right) } -\al\frac{b_n^2 u^2 (n-1)}{2}  }\int_{ \R^{n-1}}\left| \Delta_{n-1}(\lam)\right| ^\beta  \exp\left(-\frac{\al}{2}\left(b^2_n-\frac{\beta}{4\al} \right)\sum_{i=2}^{n}\left(\lam_i +\frac{b_n^2 u }{b^2_n -\frac{\beta}{4\al}}  \right)^2 \right) \mathrm{d}\lam_2...\mathrm{d}\lam_{n-1}\end{align*} which is just:
		$$Z=e^{\al\frac{b^4_n u^2}{2\left( b^2_n-\frac{\beta}{4\al}\right) } (n-1)-\al \frac{b_n^2 u^2}{2}(n-1)  } Z_{n-1,\al b^2_n-\frac{\beta}{4},\beta} .$$
		We treat the Gaussian integral term $G(t)$ in (\ref{Lambda_GZ}) with two successive change of variable and symmetry, \begin{align*}
		G(t)&=\frac{2}{\sqrt{\al}\sqrt{b^2_n-\frac{n\beta}{4\al} }} \int_{ z\geq \al\left( b^2_n-\frac{n\beta}{4\al}\right) \left( t+\frac{b^2_n u}{b^2_n-\frac{n\beta}{4\al}} \right)  } \exp\left(-\frac{z^2}{2} \right)   \mathrm{d}z
		\\&\leq \frac{2}{\al\sqrt{\al}}  \frac{e^{-\frac{\al}{2}\left( b^2_n-\frac{n\beta}{4\al }\right)^2 \left( t+\frac{b^2_n u}{b^2_n-\frac{n\beta}{4\al}} \right)^2}}{\left( b^2_n-\frac{n\beta}{4\al}\right)^\frac{3}{2} \left( t+\frac{b^2_n u}{b^2_n-\frac{n\beta}{4\al}} \right) }.
		\end{align*}
		The following classical Gaussian bound is used in the last line: $$\int_{y}^{+\infty}e^{-\frac{z^2}{2}}\mathrm{d}z\leq \frac{e^{-\frac{y^2}{2}}}{y}, \qquad y>0. $$
		Finally, (\ref{Lambda_GZ}) becomes: \begin{align*}
		\Lambda_{n,t,u}&\leq b_n^{n+\beta\frac{n(n-1)}{2}}2^{n\beta+1} e^{\al \frac{b^4_n u^2}{2\left( b^2_n-\frac{\beta}{4\al}\right) } (n-1)+\al \frac{b^4_n u^2}{2\left( b^2_n-\frac{n\beta}{4\al}\right) } -\al n\frac{b_n^2 u^2}{2}}\frac{Z_{n-1,\al b^2_n-\frac{\beta}{4},\beta}}{Z_{n,\al,\beta}}   \frac{e^{-\frac{\al}{2}\left( b^2_n-\frac{n\beta}{4\al}\right)^2 \left( t+\frac{b^2_n u}{b^2_n-\frac{n\beta}{4\al}} \right)^2}}{\al^{\frac{3}{2}}\left( b^2_n-\frac{n\beta}{4\al}\right)^\frac{3}{2} \left( t+\frac{b^2_n u}{b^2_n-\frac{n\beta}{4\al}} \right) }.
		\end{align*}
		We deal with the ratio of partition functions by (\ref{partition_function_1_eq1}) of Lemma \ref{partition_function_1}, so that the last line becomes:
	\begin{align*}
	\Lambda_{n,t,u}&\leq c_n\frac{b_n^{\beta (n-1) +1}2^{n\beta+1} }{\sqrt{2\pi}\left( b^2_n-\frac{n\beta}{4\al}\right)^\frac{3}{2}}e^{\al \frac{b^4_n u^2}{2\left( b^2_n-\frac{\beta}{4\al}\right) } (n-1)+\al \frac{b^4_n u^2}{2\left( b^2_n-\frac{n\beta}{4\al}\right) } -\al n\frac{b_n^2 u^2}{2}} \frac{e^{-\frac{\al}{2}\left( b^2_n-\frac{n\beta}{4\al}\right)^2 \left( t+\frac{b^2_n u}{b^2_n-\frac{n\beta}{4\al}} \right)^2}}{ \al\left( t+\frac{b^2_n u}{b^2_n-\frac{n\beta}{4\al}} \right) }.
	\end{align*}According to Lemma \ref{partition_function_1}, the sequence $(c_n)$ converges to $1$ hence is bounded.
	
	Lastly, with the assumption $\ds{n\beta\ll1}$ and by Taylor expansion, one can show that:
	\begin{align*}
	\al \frac{b^4_n u^2}{2\left( b^2_n-\frac{\beta}{4\al}\right) } (n-1)+\al \frac{b^4_n u^2}{2\left( b^2_n-\frac{n\beta}{4\al}\right) } -\al n\frac{b_n^2 u^2}{2}\xrightarrow[n\to\infty]{}0.
	\end{align*}
	
	We conclude that, for $M$ such that $n\beta \vee \al \leq M$, there exists a constant $C_M\in (0,+\infty)$ such that, \begin{align*}
	\Lambda_{n,t,u}&\leq C_M b_n^{\beta(n-1)-2} \frac{\exp\left( -\frac{\al}{2}\left( b^2_n-\frac{n\beta}{4\al}\right)^2 \left( t+\frac{b^2_n u}{b^2_n-\frac{n\beta}{4\al}} \right)^2\right) }{\al \left( t+\frac{b^2_n u}{b^2_n-\frac{n\beta}{4\al}} \right) }.
	\end{align*}
\end{proof}
	The same method brings also an estimate for bulk eigenvalues. We state the result as in \cite{PoissonTempLow} where a proof can be found.
	\begin{lem}\label{bulk_estimate}
		Let $M>0$ such that $\al\vee n\beta \leq M$. There exists a constant $C_M>0$ such that for any $\beta>0$, $a\in \R$, $y\in (0,1)$ and any $n\geq 1$ large enough,  \begin{align}\label{bulkestimate_eq}
P_{n,\al,\beta}\left(\left|\lam_1-a \right| \leq y \right)&\leq C_M y \exp\left( \frac{n\al \beta}{2(4\al-\beta)}a^2 \right). 
		\end{align}
	\end{lem}
	\begin{proof}
		We proceed in the same spirit as the previous estimate. We apply again Lemma \ref{lem_technique} and complete the square with: $$ -\frac{\al}{2}z_i^2+\frac{\beta(z_i-u)^2}{8}=-\frac{4\al-\beta}{8}\left( z_i+\frac{\beta u}{4\al-\beta}\right) +\frac{\al \beta u^2}{2\left( 4\al-\beta\right) }.$$ It yields an upperbound on  $\ds{	P_{n,\al,\beta}\left(\left|\lam_1 -u\right| \leq y \right)}$ involving the simpler Gaussian integral: $$\int_{\left| z_1-u\right| \leq y } \exp\left(-\frac{\al}{2} z^2_1 \right)   \mathrm{d}z_1\leq 2y \exp\left(-\frac{\al}{2}(y-u)^2 \right)\leq 2y,$$ and another ratio of partitions function treated by adapting the proof of (\ref{partition_function_1_eq1}) in Lemma \ref{partition_function_1}: $$\frac{Z_{n-1,\al -\frac{\beta}{4},\beta}}{Z_{n,\al,\beta}} \leq c_n\sqrt{ \frac{\al}{{2\pi}}}, \qquad c_n\xrightarrow[n\to\infty]{}1, \qquad 0<\al<+\infty.$$ 
\end{proof}
	\section{Poisson limit for $n\beta\ll 1$ and $\al = 1$}
	This section is devoted to the proof of Theorem \ref{main} when $\ds{\beta>0}$. Namely, we consider the extreme point process $\ds{\mathcal{P}_n=\sum_{i=1}^n \delta_{a_n(\lam_i-b_n)} }$ with $(\lam_i)_{i\leq n}\sim P_{n,\beta}$, temperature regime $\ds{ \beta:=\beta_n \ll \ff{n\log( n)}}$ and $\al=1$. Our framework is the application of Proposition \ref{tool}:
	 \begin{itemize}
		\item When $\ds{\delta_n\xrightarrow[n\to\infty]{}\delta>0}$, we consider $\ds{\mu = e^{-\frac{x}{\delta}}dx}$ and $(\lam_1,\ldots,\lam_n)$ with law: $$\rho_n \mathrm{d}\mu^{\otimes n}\left( \lam_1,\ldots,\lam_n\right)  = e^{-\frac{\al}{2} \sum_{i=1}^{n}\lam^2_i} \left| \Delta_n(\lam)\right| ^\beta e^{\ff{\delta}\sum_{i=1}^{n}\lam_i} \prod_{i=1}^{n}\mathrm{d}\lam_i  .$$
		\item When $\delta_n\gg 1$, the density $\rho_n$ equals to $P_{n,\al,\beta}$ and $\mu$ is the Lesbegue measure $\lam$ on $\R$.
	\end{itemize}
	The plan, according to Proposition \ref{tool}, is to first reformulate the correlation functions in a tractable expression, and then establish their pointwise convergence to $1$. The last step is to give an uniform upper bound which will end the proof of the theorem.
 
	The case $\beta=0$ is much simpler. Following the same steps, it does not however involve the machinery of partition functions and tail bounds. So we keep it in the last subsection.
	\subsection{Formulation of the correlation functions}
	The first step is to give a satisfying expression of the correlation function $R^n_k$. From its definition, we transpose it as product of multiple terms including an integral term $\tilde{R}^n_k$. Unlike the others, this quantity is more complicated and needs careful analysis. We express the result in the case $\ds{\delta_n\xrightarrow[n\to\infty]{}\delta>0}$ and give in a subsequent remark the analog formula for the case $\ds{\delta_n\gg 1}$.
	
	\begin{lem}\label{correlation1}
		Fix $\delta>0$. Let $\al>0$, $\beta\geq 0$ and $(\lam_1,...,\lam_n)$ distributed according to $P_{n,\al,\beta}$.
		\\For $1\leq k\leq n$, the $k$-th correlation function $\ds{R^n_k(x_1,...,x_k)}$ of the point process $\ds{\sum_{i=1}^{n}\delta_{a_n(\lam_i-b_n)}}$ is: \begin{align}\label{formulation_of_corr_funct}
		 \frac{n!}{(n-k)!}a_n^{-k -\frac{\beta}{2}k\left( k-1\right)} \left| \Delta_k(x)\right| ^\beta \frac{Z_{n-k,\al,\beta}}{Z_{n,\al,\beta}} e^{ -\frac{\al}{2} \sum_{i=1}^{k}\left( \frac{x_i}{a_n}+b_n \right)^2+\ff{\delta}\sum_{i=1}^{k}x_i+k\beta(n-k)\log(b_n)}  \tilde{R}^n_k
		\end{align}
		with the quantity $\ds{\tilde{R}^n_k:=\tilde{R}^n_k(x_1,...,x_k)}$ defined as: \begin{align}\label{formulation_of_corr_tilde}
 \int_{\R^{n-k}} \exp\left( \beta \sum_{i=1}^{k}\sum_{j=1}^{n-k} \log \left|1+\frac{x_i}{a_nb_n}-\frac{z_j}{b_n} \right| \right)  \mathrm{d}P_{n-k,\al,\beta}(z_1,...,z_{n-k}). 
		\end{align}
	\end{lem}
	\begin{proof}
		Let $\ds{\mu = e^{-\frac{x}{\delta}}dx}$, ie: $\ds{d\mu^{\otimes n}(x_1,...,x_n) = e^{-\ff{\delta}\sum_{i=1}^n x_i }dx_1...dx_n }$. Let $(\lam_1,...,\lam_n)$ distributed according to $P_{n,\al,\beta}$. By a change of variable in (\ref{P_nab}), the random vector $\ds{\left(  a_n\left( \lam_i-b_n\right) \right) _{i\leq n}}$ has joint density:
		\begin{align*}
	 \frac{a_n^{-\frac{n(n-1)}{2}\beta - n}}{Z_{n,\al,\beta}} \left| \Delta_n(\lam)\right|^\beta e^{-\frac{\al}{2} \sum_{i=1}^{n}\left( \frac{\lam_i}{a_n}+b_n \right)^2} \mathrm{d}\lam_1 \cdots \mathrm{d}\lam_n,
		\end{align*}
		which we express with respect to the measure $\mu$:
		$$\frac{a_n^{-\frac{n(n-1)}{2}\beta - n}}{Z_{n,\al,\beta}} \left| \Delta_n(\lam)\right|^\beta e^{-\frac{\al}{2} \sum_{i=1}^{n}\left( \frac{\lam_i}{a_n}+b_n \right)^2} e^{\ff{\delta}\sum_{i=1}^n \lam_i }\mathrm{d}\mu^{\otimes n}\left( \lam_1,\ldots,\lam_n\right) .$$
		Hence, using the definition (\ref{def_corfct_tool}), we deduce the $k$-th correlation function:
		\begin{align*}
		R^n_k(x_1,...,x_k) &=\frac{n!}{(n-k)!}\frac{a_n^{-n-\beta\frac{n(n-1)}{2}}}{Z_{n,\al,\beta}}e^{-\frac{\al}{2} \sum_{i=1}^{k}\left( \frac{x_i}{a_n}+b_n \right)^2+\ff{\delta}\sum_{i=1}^{k}x_i}\times  \\ & \times \int_{\R^{n-k}}  e^{-\frac{\al}{2} \sum_{i=k+1}^{n}\left( \frac{x_i}{a_n}+b_n \right)^2} \left| \Delta_n(x)\right| ^\beta e^{\ff{\delta}\sum_{i=k+1}^{k}x_i} \mathrm{d}\mu^{\otimes (n-k)}(x_{k+1},...,x_n).
		\end{align*}
		The goal is to extricate the $(x_1,...,x_k)$ from the $(x_{k+1},...,x_n)$, and extract all leading order terms.
		
		To this end, we begin with splitting the Vandermonde term: \begin{align*}
		\prod_{i<j}^n \left|x_i-x_j \right|^\beta &=\left( \prod_{1\leq i<j\leq k} \left|x_i-x_j \right|^\beta  \right) \left( \prod_{k+1\leq i<j\leq n} \left|x_i-x_j \right|^\beta\right) \left(  \prod_{i=1}^k \prod_{j=k+1}^{n}\left|x_i-x_j \right|^\beta \right) .
		\end{align*}
		Note that in the RHS, the first term has $\frac{k(k-1)}{2}$ elements, the 2nd term has $\frac{(n-k)(n-k-1)}{2}$ elements and the last term has ${k(n-k)}$ elements. \\ \\
		Therefore,
		\begin{align*}
		R^n_k(x_1,...,x_k) &=\frac{n!}{(n-k)!} \left| \Delta_k(x)\right| ^\beta \frac{a_n^{-n-\beta\frac{n(n-1)}{2}}}{Z_{n,\al,\beta}}\exp\left( -\frac{\al}{2} \sum_{i=1}^{k}\left( \frac{x_i}{a_n}+b_n \right)^2+\ff{\delta}\sum_{i=1}^{k}x_i\right) \Lambda .
		\end{align*}	
		where: $$\Lambda:= \int_{\R^{n-k}} e^{\beta \sum_{i=1}^{k}\sum_{j=k+1}^{n}\log \left|x_i-y_j \right| }e^{-\frac{\al}{2} \sum_{i=1}^{n-k}\left( \frac{y_i}{a_n}+b_n \right)^2}\left| \Delta_{n-k}(y)\right| ^\beta e^{\ff{\delta}\sum_{i=1}^{n-k}y_i} \mathrm{d}\mu^{\otimes (n-k)}(y_1,...,y_{n-k}) .$$ We introduce the law $P_{n,\al,\beta}$ in the latter quantity. The change of variable $y=a_n(z-b_n)$ and little computation give:
		\begin{align*}
		\Lambda &=a_n^{n-k + \beta \frac{(n-k-1)(n-k)}{2}} \int_{\R^{n-k}} e^{\beta \sum_{i=1}^{k}\sum_{j=1}^{n-k}\log \left|x_i-a_n(z_j-b_n) \right| }e^{-\frac{\al}{2} \sum_{i=1}^{n-k}z_j^2} \prod_{1\leq i<j\leq n-k} \left|z_i-z_j \right|^\beta \prod_{i=1}^{n-k}\mathrm{d}z_i 
		\\&=a_n^{n-k + \beta \frac{(n-k-1)(n-k)}{2}+k\beta(n-k)}Z_{n-k,\al,\beta }e^{k\beta(n-k)\log(b_n)}\times  \\&\qquad\qquad\qquad  \times 
		\int_{\R^{n-k}} e^{\beta \sum_{i=1}^{k}\sum_{j=1}^{n-k}\log \left|1+\frac{x_i}{a_nb_n}-\frac{z_j}{b_n} \right| } \mathrm{d}P_{n-k,\al,\beta}(z_1,.,z_{n-k})
		.
		\end{align*}
		Thus the claim follows.
	\end{proof}
\begin{lem}
Assume $\ds{\delta_n\gg 1}$ and $\mu$ to be the Lebesgue measure $\lam$ on $\R$. In this case, the $k$-th correlation function is given by: \begin{align}\label{corr_function_sans_delta}
		 R^n_k&=\frac{n!}{(n-k)!}a_n^{-k -\frac{\beta}{2}k\left( k-1\right)} \left| \Delta_k(x)\right| ^\beta \frac{Z_{n-k,\al,\beta}}{Z_{n,\al,\beta}} e^{ -\frac{\al}{2} \sum_{i=1}^{k}\left( \frac{x_i}{a_n}+b_n \right)^2+k\beta(n-k)\log(b_n)}  \tilde{R}^n_k
\end{align} with $\tilde{R}^n_k$ defined in (\ref{formulation_of_corr_tilde}). 
\end{lem}
\begin{proof}
 The proof goes along the same lines as the demonstration of Lemma \ref{correlation1} except that all the $\delta$-dependent terms vanish.
\end{proof}
	\subsection{Pointwise convergence of the correlation functions}
		The goal of this section is to establish the pointwise convergence $\ds{R^n_k(x_1,...,x_k)\xrightarrow[n\to\infty]{}1}$ for any fixed $0<\delta<+\infty$, $1\leq k<+\infty$ and $(x_1,...,x_k)\in \R^k$ under the following hypothesis: $$\beta \ll \ff{n\log(n)},\qquad \al=1, \qquad \delta_n = \delta+o(1) \text{ or } \delta_n\gg 1  .$$
We have already shown the ratio of partition functions converges to $\ds{\left( 2\pi\right)^{-\frac{k}{2}} }$ in (\ref{partition_function_01}) of Lemma \ref{partition_function_0}. The other terms are easily handable, so we begin by proving that the term $\tilde{R}^n_k$ converges to $1$. To this end, we proceed by double inequality.

	\begin{lem}\label{tilde_convergence}
		Let any $(\delta_n)$ positive real sequence such that $\ds{ {\log(\delta_n)}\ll \log(n) }$. Assume $\ds{\beta\ll \ff{n\log(n)}}$. Fix an integer $1\leq k<+\infty$ and $(x_1,...,x_k)\in \R^{k}$, let \begin{align*}
		\tilde{R}^n_k:=\tilde{R}^n_k(x_1,...,x_k) &= \E_{P_{n-k,\beta} }\left(\exp\left( \beta \sum_{i=1}^{k}\sum_{j=1}^{n-k} \log \left|1+\frac{x_i}{a_nb_n}-\frac{\lam_j}{b_n} \right| \right) \right) .
		\end{align*}
		Then the following pointwise convergence holds: $$R^n_k(x_1,...,x_k)\xrightarrow[n\to\infty]{}1.$$ 
		\end{lem}
	\begin{proof}
		We begin by showing that $\ds{\limsup_{n\infty} \tilde{R}^n_k(x_1,...,x_k)\leq 1}$. Applying the bound (\ref{bound_technique}) of Lemma \ref{lem_technique}, and with little computation, we get:
		\begin{align*}
		\tilde{R}^n_k &\leq 2^{kn\beta}\exp\left( \frac{n\beta}{8} \sum_{i=1}^{k} \left|1+\frac{x_i}{a_nb_n} \right|^2\right)  \frac{Z_{n-k,1-\frac{k\beta}{4b^2_n},\beta}}{Z_{n-k,\beta}}.
		\end{align*}
		With the assumption $n\beta\ll 1$ and $k<+\infty$, it is enough to show this ratio of partition functions converges to $1$, which is provided by (\ref{partition_function_02}) of Lemma \ref{partition_function_0}.
		
		Hence, our task is now to show that $$\liminf_{n\infty} \tilde{R}^n_k(x_1,...,x_k)\geq 1.$$
	
		Since $\exp$ is convex, by Jensen inequality, and exchangeability, it is enough to show that for any $x\in \R$ fixed, $$ \beta (n-k)  \E_{P_{n-k,\beta} }\left| \log \left|1+\frac{x}{a_nb_n}-\frac{\lam_1}{b_n} \right|  \right| \xrightarrow[n\to\infty]{}0.$$
		Since $1\leq k<+\infty$ is also fixed, it is enough to show that for $x\in \R$ fixed,
		$$n \beta  \E_{P_{n,\beta} }\left| \log \left|1+\frac{x}{a_nb_n}-\frac{\lam_1}{b_n} \right|  \right| \xrightarrow[n\to\infty]{}0.$$	\end{proof} The following result will complete our proof:
		
		\begin{lem}\label{jensen}
			Let any $(\delta_n)$ positive real sequence such that $\ds{ {\log(\delta_n)}{}\ll \log(n) }$. Assume $\ds{\beta\ll \ff{n\log(n)}}$. Fix $x\in \R$, then: $$n \beta  \E_{P_{n,\beta} }\left| \log \left|1+\frac{x}{a_nb_n}-\frac{\lam_1}{b_n} \right|  \right| \xrightarrow[n\to\infty]{}0.$$
		\end{lem}
		\begin{proof}[Proof of Lemma \ref{jensen}]
			From the identity $$\E\left|X \right|= \int_{0}^{+\infty}\p\left(\left|X\right|\geq t \right) \mathrm{d}t, $$ setting $\ds{u:=1+\frac{x}{a_n b_n}}$, removing the absolute value, and by a change of variable, we have:
			\begin{align*}
			\E_{P_{n,\beta} }\left| \log \left|1+\frac{x}{a_nb_n}-\frac{\lam_1}{b_n} \right|  \right| &= \int_{0}^{+\infty} P_{n,\beta}\left( \left| \log \left|u-\frac{\lam_1}{b_n} \right|  \right|  \geq t\right) \mathrm{d}t 
			\\&= \int_{1}^{+\infty}\ff{y} P_{n,\beta}\left(   \left|u-\frac{\lam_1}{b_n} \right|    \geq  y\right) \mathrm{d}y +\int_{0}^{1}\ff{y} P_{n,\beta}\left(   \left|u-\frac{\lam_1}{b_n} \right|    \leq y\right) \mathrm{d}y.
			\end{align*}
			Next, we show that both integrals converge to $0$. We set: $$\Lambda_1:= \int_{1}^{+\infty}\ff{y} P_{n,\beta}\left(   \left|u-\frac{\lam_1}{b_n} \right|    \geq  y\right) \mathrm{d}y,\qquad \Lambda_2:= \int_{0}^{1}\ff{y} P_{n,\beta}\left(   \left|u-\frac{\lam_1}{b_n} \right|    \leq y\right) \mathrm{d}y.$$
			Let's treat the term $\Lambda_2$.
			\\ \\ Since $\al=1$ and $n\beta \ll 1$, we can find $M>0$ satisfying the assumption $\al\vee n\beta \leq M$ of Lemma \ref{bulk_estimate}. 
			Hence, with $\al=1$, $\ds{a=b_n+\frac{x_i}{a_n}=b_n u}$ and $\ds{0<y\leq1}$ in the bulk estimate (\ref{bulkestimate_eq}), there exists a constant $C>0$ independent of $n,k,y$ such that:
			\begin{align*}
			P_{n,\beta}\left(\left|u-\frac{\lam_j}{b_n} \right| \leq y  \right) &\leq  C b_ny \exp\left( \frac{ n\beta}{8\left( 1-\frac{\beta}{4}\right) }b^2_n u^2 \right).
			\end{align*}
			It follows that: \begin{align*}
			0\leq (n-k)\beta\Lambda_2 &\leq C n\beta  b_n \exp\left(  \frac{n\beta}{8\left( 1-\frac{\beta}{4}\right) }b_n^2 u^2 \right).
			\end{align*}
			The latter term goes to $0$ if and only if $\ds{\beta \ll \frac{1}{n\log(n)} }$. This is an explicit circumstance where we need to strengthen the restriction on $\ds{\beta\ll \ff{n}}$ alluded in Remark \ref{extrafactor}.

			Regarding the term $\Lambda_1$, we use the top eigenvalue estimate (\ref{top_estimate_eq}) of Lemma \ref{top_estimate} with $\al=1$:
			\begin{align*}
			\Lambda_1 &\leq C_M b_n^{\beta(n-1)-2}\int_1^{+\infty} \frac{e^{-\frac{1}{2}\left( b^2_n-\frac{n\beta}{4}\right)^2 \left( t+\frac{b^2_n u}{b^2_n-\frac{n\beta}{4}} \right)^2}}{t \left( t+\frac{b^2_n u}{b^2_n-\frac{n\beta}{4}} \right) }\mathrm{d}t
			\end{align*}
			The Lebesgue's dominated convergence theorem implies that integral term converges to $0$ which leads to $n\beta \Lambda_1 \ll 1$.
		\end{proof}

	We are ready to achieve the goal of this section:
	\begin{Prop}\label{prop_conv}
		Assume $\al=1$ and $\ds{ \beta \ll \ff{n\log( n)}}$. Let $(\delta_n)$ a positive sequence and the modified Gaussian scaling:$$b_n:=\sqrt{2\log (n)}-\frac{\log\log (n) +2\log(\delta_n)+ \log (4\pi)}{2\sqrt{2\log (n)}},\qquad a_n:= \delta_n\sqrt{2\log (n)} .$$ Fix an integer $1\leq k<+\infty$ and $(x_1,...,x_k)\in \R^k$. In the two cases: $$ a) \quad \delta_n \xrightarrow[n\to\infty]{}\delta>0  \text{ and } \mu=e^{-\frac{x}{\delta}} \mathrm{d}x\qquad b)\quad  \delta_n \gg 1 \text{ with } {\log(\delta_n)}{}\ll\sqrt{\log(n)} \text{ and }\mu=\lam,$$
		the following pointwise convergence holds: $$R^n_k(x_1,...,x_k)\xrightarrow[n\to\infty]{}1. $$
	\end{Prop}
	\begin{proof} Since both cases share a lot in common, we proceed to the proof assuming case a) and then only mention the deviations for the second case.
		When $\al=1$, the formula (\ref{formulation_of_corr_funct}) of Lemma \ref{correlation1} gives:\begin{align}\label{aux1}
		R^n_k&= \frac{n!}{(n-k)!}a_n^{-k -\frac{\beta}{2}k\left( k-1\right)} \left| \Delta_k(x)\right| ^\beta \frac{Z_{n-k,\beta}}{Z_{n,\beta}}e^{-\frac{1}{2} \sum_{i=1}^{k}\left( \frac{x_i}{a_n}+b_n \right)^2+\ff{\delta}\sum_{i=1}^{k}x_i+k\beta(n-k)\log(b_n)}  \tilde{R}^n_k
		\end{align}
		with $$\tilde{R}^n_k:=\tilde{R}^n_k(x_1,...,x_k)= \int_{\R^{n-k}} \exp\left( \beta \sum_{i=1}^{k}\sum_{j=1}^{n-k} \log \left|1+\frac{x_i}{a_nb_n}-\frac{z_j}{b_n} \right| \right)  \mathrm{d}P_{n-k,\beta}(z_1,...,z_{n-k}). $$
		We already proved that $\ds{\tilde{R}^n_k(x_1,...,x_k)\xrightarrow[n\to\infty]{}1 }$ in Lemma \ref{tilde_convergence}. Hence, we are reduced to show that:
		$$ \frac{n!}{(n-k)!}a_n^{-k -\frac{\beta}{2}k\left( k-1\right)} \left| \Delta_k(x)\right| ^\beta  \frac{Z_{n-k,\beta}}{Z_{n,\beta}} e^{-\frac{1}{2} \sum_{i=1}^{k}\left( \frac{x_i}{a_n}+b_n \right)^2+\ff{\delta}\sum_{i=1}^{k} x_i}e^{k\beta(n-k)\log(b_n)} \xrightarrow[n\to\infty]{}1. $$
		For each term of the latter, we have the following asymptotics as $k<+\infty$ is fixed: 
			$$ \frac{n!}{(n-k)!}=(1+o(1))n^k, \quad 	a_n^{-\beta\frac{k(k-1)}{2}-k}= \exp\left( -k \log a_n \right) \left( 1+o(1)\right)  $$ 
		$$	\frac{Z_{n-k,\beta}}{Z_{n,\beta}} =  \left( 2\pi \right)^{-\frac{k}{2} }+o(1) , \quad \exp\left( k\beta(n-k)\log b_n\right)  =\exp\left( \frac{k}{2}n\beta \log\log (n)\right)   \left( 1+o(1)\right)  $$
							$$\Delta_k(x_1,...,x_k)^\beta = \prod_{i<j}^{k} \left|{x_i}{}-{x_j} \right|^\beta=\exp\left(\beta \sum_{i<j}^{k}\log \left|{x_i}{}-{x_j} \right| \right)=1+o(1). $$
		Moreover, expanding the square and since $\ds{\frac{b_n}{a_n}=\ff{\delta}+o(1)}$,
		\begin{align}\label{aux2}
		\exp\left( -\frac{1}{2} \sum_{i=1}^{k}\left( \frac{x_i}{a_n}+b_n \right)^2\right) &= \exp\left(- \frac{kb_n^2}{2} -\ff{\delta}\sum_{i=1}^{k}x_i +o(1)\right) .
		\end{align}
		So putting everything together, \begin{align*}
		R^n_k(x_1,...,x_k)&=\exp\left(  k\left(\log (n) -\log( a_n) -\ff{2}b^2_n +n\beta \log (b_n) -\ff{2}\log 2\pi   \right)   \right)   \left( 1+o(1)\right)
		\\&:= e^{k \Lambda_n} \left( 1+o(1)\right) .
		\end{align*}
		Thanks to the following asymptotics:
		$$\log (a_n) = \ff{2}\log (2) + \ff{2}\log\log (n) + \log(\delta) +o(1) $$
		\begin{align}\label{b^2_n}
b^2_n = 2\log(n) - \log\log(n) -2\log(\delta) -\log(4\pi) +o(1),\quad n\beta \log(b_n) \ll 1 
		\end{align}
		the computation of $\Lambda_n$ shows that only negligible terms remain, the others canceling each other out. In relation with Remark \ref{extrafactor}, let us point out that $\ds{n\beta \log(b_n)\ll 1}$ in (\ref{b^2_n}) occurs when $\ds{n\beta {\log\log(n)}\ll 1}$ which is naturally covered by our hypothesis $\ds{n\beta\log(n) \ll 1}$.
		
		Let $(x_1,...,x_k)\in \R^k$ and consider case b). The quantity $\ds{R^n_k(x_1,...,x_k)}$ is given by formula (\ref{corr_function_sans_delta}). Meanwhile, the cross term from the square expansion in the LHS of (\ref{aux2}) vanishes as $\delta_n \gg 1 $ implies $a_n\gg b_n$. Thus its corresponding term in the RHS of (\ref{aux2}) also disappeares. By the asymptotics used previously, when $n\to +\infty$, the quantity $\ds{R^n_k(x_1,...,x_k)}$ is equivalent to the same $\ds{\exp\left( k\Lambda_n\right) }$ as found previously.
		Now, the difference with (\ref{b^2_n}) lies in a cross term in the expansion:
		\begin{align*}
		b^2_n &= 2\log(n) +\frac{\log^2(\delta_n)}{2\log(n)}- \log\log(n) -2\log(\delta_n) -\log(4\pi) +o(1).
		\end{align*}
		Finally, after some cancelations in the computation,
		\begin{align*}
		\Lambda_n&=-\frac{\log^2(\delta_n)}{2\log(n)}+o(1).
		\end{align*}
		The latter quantity converges to $0$ under the growth hypothesis on $(\delta_n)$.
		
	\end{proof}
	\subsection{Uniform upper-bound on the correlation functions}
	The goal of this section is to provide an uniform upper bound for the correlation functions. It constitutes the second hypothesis in the main tool required to show Poisson convergence. We state the result regardless of the measure $\mu$ chosen in Proposition \ref{tool}. Indeed, the correlation functions differs slightly and the proof is not impacted. For these reasons, we will only show the result in the case a) of Proposition \ref{prop_conv}.
\\
	\begin{lem}\label{upperbound}
		Assume $\al=1$, $\ds{\delta_n\xrightarrow[n\to\infty]{}\delta>0}$ and $\ds{\mu=e^{\frac{x}{\delta}}dx}$. Let $K\subset \R$ compact. There exists a constant $\Theta_K>0$ such that for any $n\geq 1$ large enough, any integer $1\leq k\leq n$ and any $(x_1,...,x_k)\in K^k$, $$ {R}^n_{k}(x_1,...,x_k)\leq \Theta^k_K.$$
	\end{lem}
\begin{proof} Assume $\al=1$.
	Let $K\subset \R$ compact. We can always find $M:=M_K>1$ such that $\ds{\forall x\in K,\left| x\right| \leq M }$. Let $k\leq n$ and $x_1,...,x_k\in K$.
	Note that $(\delta_n)$ converges to $\delta>0$ hence is bounded. 
	
	Our goal is to bound in terms of the quantity $M$ the formula (\ref{formulation_of_corr_funct}) of the correlation functions $\ds{R^n_k(x_1,...,x_k)}$ given by Lemma \ref{correlation1}.
	
	First, we bound by elementary means the simple terms. The leading order terms will cancel each other in the computation. Then, we tackle the integral term $\tilde{R}^n_k$ by comparing it to some ratio of partition functions.
	
		We begin to notice that, according to (\ref{partition_function_32}) of Lemma \ref{partition_function_3}, the ratio of partition functions in (\ref{formulation_of_corr_funct}) is bounded by $\ds{\left( \frac{2}{\pi}\right) ^{\frac{k}{2}}}$.
	The Vandermonde determinant is easily treated. Since $\ds{\frac{\beta(k-1)}{2}\leq  1}$ for any $1\leq k\leq n$, one has:
	$$ \left| \Delta_k(x)\right| ^\beta = \prod_{i<j}^{k}\left| x_i -x_j \right|^\beta \leq  M^{\beta\frac{k(k-1)}{2}} =\left( M^{\frac{\beta(k-1)}{2}}\right)^k \leq M^k .$$
	Also, $$a_n^{-k -\frac{\beta}{2}k\left( k-1\right)} \leq a^{-k}_n =\exp\left( -k\left(\log(\delta_n)+\ff{2}\log\log (n)+\ff{2}\log(2) \right)  \right) .$$
	Since $\ds{ \forall x\in \R, 1+x\leq e^x}$, we treat the combinaison term as follows, for $1\leq k \leq n$: $$\frac{n!}{n^k(n-k)!}=\prod_{i=0}^{k-1} \left (1-\frac{i}{n}\right) \le 
	\prod_{i=0}^{k-1} \exp\left (-\frac{i}{n}\right)= 
	\exp\left(-\frac{(k-1)k}{2n}\right)\leq 1. $$
	Let's now study the exponential terms:
	\begin{align}\label{exp_term_to_bound}
	\exp\left(-\frac{1}{2}\sum_{i=1}^{k}\left( \frac{x_i}{a_n}+b_n\right)^2  \right)
	&\leq \exp\left(-\frac{k}{2}b^2_n- \frac{b_n}{a_n} \sum_{i=1}^{k}x_i \right)\leq \exp\left(-\frac{k}{2}b^2_n+kc_\delta M\right).
	\end{align}
	We used the fact that, since $(\delta_n)$ is bounded, there exists $c_\delta>0$ such that for any $n\geq 1$: $$\frac{b_n}{a_n}\leq \frac{1}{\delta_n}\leq c_\delta. $$
	For $1\leq k\leq n$, since $\ds{b_n\leq \sqrt{2\log(n)}}$ and $\ds{n\beta\ll 1}$,\begin{align*}
	k\beta(n-k)\log(b_n) &\leq  \frac{k}{2}n\beta \log\log(n)+k.
	\end{align*}
	Hence, using the definition of $b_n$ in (\ref{exp_term_to_bound}), the leading order terms in (\ref{formulation_of_corr_funct}) cancel:
	\begin{align*}
	R^n_k(x_1,...,x_k)&\leq M^{k}  \exp\left( k+3k\log(2)+kc_\delta M\right) \tilde{R}^n_k(x_1,...,x_k).
	\end{align*}
	It remains to bound the term $\tilde{R}^n_k$. \\ \\
	Applying the bound (\ref{bound_technique}) of Lemma \ref{lem_technique} on the formula (\ref{formulation_of_corr_tilde}), we get:
	\begin{align}\label{rtilde_1}
	\tilde{R}^n_k(x_1,...,x_k) &\leq 2^{kn\beta}e^{\frac{n\beta}{8} \sum_{i=1}^{k} \left|1+\frac{x_i}{a_nb_n} \right|^2} \frac{Z_{n-k,1-\frac{k\beta}{4b^2_n},\beta}}{Z_{n-k,\beta}}.
	\end{align}
	By inequality (\ref{partition_function_31}) of Lemma \ref{partition_function_3}, the ratio of partition functions is bounded by $4^k$.
	
	Moreover,
	$$\exp\left( \frac{n\beta}{8} \sum_{i=1}^{k} \left|1+\frac{x_i}{a_nb_n} \right|^2\right)  \leq \exp\left( \ff{8}\sum_{i=1}^{k}\left( 1+\frac{x^2_i}{a^2_nb^2_n}+\frac{2x_i}{a_nb_n}\right)  \right) \leq \exp\left( {\frac{k}{8}+\frac{k}{8}M^2 +\frac{k}{4} M}\right) .$$
	Thus, (\ref{rtilde_1}) becomes:	\begin{align*}
	\tilde{R}^n_k(x_1,...,x_k) &\leq 2^{3k} \exp\left( \frac{k}{8}+\frac{k}{8}M^2+\frac{k}{4} M\right) .
	\end{align*}
	The claim follows with: $$\Theta_K= \exp\left( {\frac{9}{8}+6\log(2)+c_\delta M+\ff{8}\left( M^2 + 2 M\right)+\log\left( M\right)  } \right) >0.$$
\end{proof}
	\subsection{The Gaussian case: $\beta=0$}
	In this last subsection, we derive our result on homogeneous limiting Poisson process in the purely Gaussian case $\beta = 0$. Although the correlation functions method also applies (as we did for the inhomogeneous case when $\beta=0$), it turns out that the classical method from EVT provides a better regime for the perturbation $(\delta_n)$. We formulate the result and prove it.
	\begin{Prop}
		Let $(\lam_i)_{i\leq n}$ an i.i.d. sequence of $\mathcal{N}(0,1)$. Let $\delta_n\gg 1$ such that $\ds{{\log(\delta_n)}{{}}\ll \log(n)}$, and: $$ a_n = \delta_n\sqrt{2\log(n)}$$ $$b_n = \sqrt{2\log(n)}-\ff{2}\frac{\log\log(n)+2\log(\delta_n)+\log(4\pi)}{\sqrt{2\log(n)}}.$$
		Then, the point process $\ds{\sum_{i=1}^{n}\delta_{a_n(\lam_i-b_n)}}$ converges to a Poisson point process on $\R$ with intensity $1$. 
	\end{Prop}
	\begin{proof}
		We set $\ds{\phi_n(x)=\frac{x}{a_n}+b_n.}$
		Since we consider a collection of $n$ i.i.d. random variables and a homogeneous limiting Poisson process, that is with intensity proportional to $d\lam$ where $\lam$ is the Lesbegue measure on $\R$, it is enough \cite[Th 7.1]{Coles} to show that for any $x<y$, $$\Lambda:=n\left( \p\left( \lam_1\geq \phi_n(x)\right)-\p\left( \lam_1\geq \phi_n(y)\right) \right) \xrightarrow[n\to\infty]{} y-x. $$
 		By Mill's ratio, we know that for any $u\gg 1$, $$ \p\left( \lam_1\geq u\right) = \frac{\exp\left( -\frac{u^2}{2}\right) }{u\sqrt{2\pi}} \left( 1+o(1)\right) .$$
		Under the hypothesis $\ds{\log(\delta_n)\ll \log(n)}$, one has $\ds{b_n \sim \sqrt{2\log(n)}}$, hence $\ds{\phi_n(x)\sim \sqrt{2\log(n)}}$. \\ \\ We get: \begin{align*}
		\Lambda &= \frac{n}{\sqrt{2\pi}}\left(\frac{e^{-\frac{\phi_n(x)^2}{2}}}{\phi_n(x)}-\frac{e^{-\frac{\phi_n(y)^2}{2}}}{\phi_n(y)} \right) \left( 1+o(1)\right) 
		\\&= \frac{ne^{-\frac{\phi_n(x)^2}{2}}}{\sqrt{2\log(n)} \sqrt{2\pi}}\left(1-e^{\frac{\phi_n(x)^2-\phi_n(y)^2}{2}} \right) \left(1+o(1) \right) .
		\end{align*}
		A little computation gives:
		$$\frac{\phi_n(x)^2-\phi_n(y)^2}{2} = \frac{x^2-y^2}{4\delta^2_n \log(n)}+\frac{x-y}{\delta_n} - \frac{(x-y)\log\log(n)}{2\delta_n\log(n)}-\frac{(x-y)\log(\delta_n)}{\delta_n \log(n)}-\frac{(x-y)\log(4\pi)}{2\delta_n\log(n)}. $$
		The highest order term is $\ds{\frac{x-y}{\delta_n}}$. Indeed,
		$$\frac{\log(\delta_n)}{\delta_n \log(n)}\ll \ff{\delta_n}\iff \log(\delta_n)\ll \log(n) \quad \text{which is true.} $$
		We deduce that: \begin{align*}
		\Lambda &=\frac{ne^{-\frac{\phi_n(x)^2}{2}}}{\sqrt{2\log(n)} \sqrt{2\pi}}\left(\frac{y-x}{\delta_n} \right) \left(1+o(1) \right) .
		\end{align*}
		To conclude, we compute: $$\frac{ne^{-\frac{\phi_n(x)^2}{2}}}{\sqrt{2\log(n)} \sqrt{2\pi}} = \delta_n\left( 1+o(1)\right) . $$ 
	\end{proof}
		\bibliographystyle{plain}
	\bibliography{refpoisson}
	\Addresses
\end{document}